\documentclass[12pt]{article}

\usepackage{tikz}
\usetikzlibrary{arrows}
\usepackage{hyperref}
\usepackage{subcaption}
\usepackage{enumerate}
\usepackage{cite}
\usepackage{comment}
\usepackage{color}
\usepackage{amsmath}
\numberwithin{equation}{section}
\usepackage{mathtools}
\usepackage{physics}
\usepackage{cases}
\usepackage{listings}
\usepackage{amsthm}
\theoremstyle{definition}
\newtheorem{thm}{Theorem}[section]
\newtheorem{prop}[thm]{Proposition}
\newtheorem{lem}[thm]{Lemma}

\newtheorem{defn}[thm]{Definition}
\newtheorem*{pfmain}{Proof of Theorem \ref{maintheorem}}
\newtheorem*{pfstrong}{Proof of Theorem \ref{strongsolution}}
\newtheorem*{pfglobal}{Proof of Theorem \ref{globalsol}}
\newtheorem{rmk}[thm]{Remark}

\newtheorem*{ack}{Acknowledgements}
\theoremstyle{plain}
\usepackage{amssymb}
\newcommand{\al}{\alpha}
\newcommand{\gam}{\gamma}
\newcommand{\del}{\delta}
\newcommand{\ep}{\varepsilon}

\newcommand{\lam}{\lambda}
\newcommand{\sig}{\sigma}

\newcommand{\Del}{\Delta}
\newcommand{\Gam}{\Gamma}

\newcommand{\Om}{\Omega}
\newcommand{\C}{\mathbb{C}}

\newcommand{\N}{\mathbb{N}}

\newcommand{\R}{\mathbb{R}}

\newcommand{\Z}{\mathbb{Z}}

\newcommand{\cF}{\mathcal{F}}

\newcommand{\cM}{\mathcal{M}}
\newcommand{\cN}{\mathcal{N}}
\newcommand{\cP}{\mathcal{P}}
\newcommand{\cS}{\mathcal{S}}

\renewcommand{\leq}{\leqslant}
\renewcommand{\geq}{\geqslant}
\newcommand{\pl}{\partial}

\newcommand{\BMnorm}[5]{\norm{#1\mid N^{#2}_{#3,#4,#5}}}
\newcommand{\Mnorm}[3]{\norm{#1\mid M^{#2}_{#3}}}
\newcommand{\inftynorm}[1]{\norm{#1\mid L^{\infty}}}
\newcommand{\xtnorm}[1]{\norm{#1\mid X_{T}}}
\newcommand{\onorm}[2]{\norm{#1\mid #2}}

\begin{document}
\title{Existence of solutions for time fractional semilinear parabolic equations \\in Besov--Morrey spaces}

\author{Yusuke Oka and Erbol Zhanpeisov}
\date{\today}

\maketitle


\begin{abstract}
  We consider the Cauchy problem for a time fractional semilinear heat equation
  \begin{align*}
    \begin{cases}
      ^{C}\pl^{\al}_{t}u=\Del u+|u|^{\gam -1}u, & \quad x\in\R^N,\: t>0,\\
      u(x,0)=\mu(x), & \quad x\in\R^N,
    \end{cases}\tag{P}
  \end{align*}
  where $0<\al<1,\: \gam>1,\: N\in\Z_{\geq 1}$ and $\mu(x)$ belongs to inhomogeneous/homogeneous Besov--Morrey spaces. The fractional derivative $^{C}\pl^{\al}_{t}$ is interpreted by Caputo sense.
  We present sufficient conditions for the existence of local/global-in-time solutions to problem (P).
  Our results cover all existing results in the literature and can be applied to a large class of initial data.
\end{abstract}
\vspace{25pt}
\noindent Addresses:

\smallskip
\noindent
Y.~O.:  Graduate School of Mathematical Sciences, The University of Tokyo,\\
\qquad\,\,\, 3-8-1 Komaba, Meguro-ku, Tokyo 153-8914, Japan. \\
\noindent
E-mail: {\tt oka@ms.u-tokyo.ac.jp}\\

\smallskip
\noindent E.~Z.:  Okinawa Institute of Science and Technology,\\
\qquad\,\,\, 1919-1 Tancha, Onna-son, Kunigami-gun, Okinawa 904-0495, Japan.\\
\noindent
E-mail: {\tt erbol.zhanpeisov@oist.jp}\\
\vspace{20pt}

\vspace{3pt}

\vspace{3pt}

\newpage
\section{Introduction and main results}
We are interested in the existence of a solution to the time fractional Cauchy problem
\begin{align}\label{eq}
  \begin{cases}
    ^{C}\pl^{\al}_{t}u=\Del u+|u|^{\gam -1}u, & \quad x\in\R^N,\: t\in\left] 0,T \right[,\\
    u(x,0)=\mu(x), & \quad x\in\R^{N},
  \end{cases}
\end{align}
where $T>0$, $0<\al<1,\: \gam>1,\: N\geq 1$ and $\mu(x)$ is a tempered distribution on $\R^{N}$.
The symbol $^{C}\pl^{\al}_{t}$ denotes the Caputo derivative of order $\al$. For an absolute continuous function $y(t)$, its Caputo derivative is defined by
\begin{align*}
  ^{C}\pl^{\al}_{t}y(t)&\coloneqq\frac{1}{\Gam(1-\al)}\int_{0}^{t}(t-\tau)^{-\al}\frac{\pl}{\pl\tau}y(\tau)\,d\tau,
\end{align*}
where $\Gam$ is the usual Gamma function.

Recently, fractional calculus has been used in various fields to describe real-world phenomena which can not be explained by usual Brownian motion.
To give an example, time fractional diffusion equations were proposed as a macroscopic model for anomalous diffusion in heterogeneous media as pointed out by \cite{AdamsGelhar}.
For mathematical treatments, see \cite{CNYY,GigaMitakeSato,SakamotoYamamoto}.

In this paper, we consider the time local and global well-posedness of problem \eqref{eq} with initial data $\mu(x)$ belonging to Besov--Morrey spaces.
With the aid of these spaces, we construct local/global-in-time solutions to problem \eqref{eq} for wider class of initial data than has been treated in the literature.

There have been a lot of mathematical works on the solvability of problem \eqref{eq} with $\al=1$, see
e.g.,\cite{BarasPierre85,Fujita66,HisaIshige2018,Hayakawa,Miyamoto2021,Sugitani75,TW2014,Umakoshi,Weissler80,Weissler81,Zhanpeisov}.
Without going into the details of these results, it suffices to remark that
\begin{itemize}
  \item Let $\gam>1+2/N$. Then problem \eqref{eq} is well-posed in $L^{p}(\R^{N})$ for all $p\geq q_{c}\coloneqq{N(\gam-1)}/{2}\:\:(>1)$,
  \item Let $\gam=1+2/N$. Then problem \eqref{eq} is well-posed in $L^{p}(\R^{N})$ for all $p>1$,
  \item Let $1<\gam<1+2/N$. Then problem \eqref{eq} is well-posed in a set of distributions including Radon measures.
\end{itemize}
Note that in the case of $\gam=1+2/N$, problem \eqref{eq} with $\al=1$ is not well-posed in $L^1$ (see for example, \cite{BrezisCazenave,CelikZhou,HisaIshige2018}).
Among others, Kozono and Yamazaki \cite{KozonoYamazaki} introduced Besov--Morrey spaces and obtained the time local well-posedness of problem \eqref{eq} for a large class of initial data.
Kozono and Yamazaki \cite{KozonoYamazaki} also proved the time global well-posedness of problem \eqref{eq} in the case of $\gam>1+2/N$. For $\gam\leq 1+2/N$, there are no nontrivial nonnegative global solutions of problem \eqref{eq} (see for example, \cite{HisaIshige2018,Hayakawa,Sugitani75}).
The aim of the present paper is to extend their results to the case of time fractional parabolic equations.

For time fractional semilinear parabolic equations, the time local/global well-posedness was discussed in \cite{AlmeidaPrecioso,C-N,Nabti,Suzuki2022,WCX}
(Also, there is a good reference book \cite{GWBook}).
Let $\Omega$ be a smooth bounded domain in $\R^{N}$.
Recently, Ghergu, Miyamoto and Suzuki \cite{GMS} consider the following Cauchy--Dirichlet problem
\begin{align}
  \begin{cases}
    ^{C}\pl^{\al}_{t}u=\Del u+|u|^{\gam-1}u, & x\in\Om,\,t>0,
    \\
    u(x,t)=0, & x\in\pl\Om,\,t>0,
    \\
    u(x,0)=\mu(x)\in L^{p}(\Om), & x\in\Om
  \end{cases}
\end{align}
for some $p\in\left[1,\infty\right]$.
They consider the integral equation
\[
u(t)=P_{\al}(t)\mu+
\int_{0}^{t}(t-\tau)^{\al -1}S_{\al}(t-\tau)|u|^{\gam -1}u(\tau)\,d\tau
\]
and proved the following:
\begin{itemize}
  \item Let $\gam>1+{2}/{N}$. Then problem \eqref{eq} is well-posed in $L^{p}(\Omega)$ for all $p\geq q_{c}\coloneqq{N(\gam-1)}/{2}\:\:(>1)$,
  \item Let $1<\gam\leq 1+{2}/{N}$. Then problem \eqref{eq} is well-posed in $L^{p}(\Omega)$ for all $p\geq 1$.
\end{itemize}
For the precise definition of the symbol, see Definition \ref{def:timefracheatkernel} (Since their work was done in a bounded domain $\Om$, you must read the Gauss kernel $\cF^{-1
}\exp(-|\xi|^{2}t)\cF$ appearing in the definition as the Dirichlet heat kernel on $\Om$).
Note that the problem \eqref{eq} has a local-in-time solution in the case of $(\gam,p)=(1+2/N,1)$ which is a striking difference compared to the classical case $\al=1$ (see also \cite{Kojima}).
Moreover, they obtained that if $\gam\geq 1+2/N$, there exists a global-in-time mild solution with sufficiently small $\mu$ in the sense of $L^{q_{c}}(\Omega)$.

Their results are based on $L^{p}$-$L^{q}$ estimates of operators $P_{\al}(t)$ and $S_{\al}(t)$, but this is not adequate to deal with (general) distributions as initial data.
In this paper, we establish new decay estimates of $P_{\al}(t)$ and $S_{\al}(t)$ in Besov--Morrey spaces to give sufficient conditions for the existence of local/global-in-time solutions to problem \eqref{eq} with wider class of initial data, which includes not only all previous results (see Remark \ref{rmk:optimality} and Remark \ref{highfreq}) but also distributions other than Radon measures.
Furthermore, we prove the global-in-time existence of solutions with weaker restrictions on $\gam$ than $\gam\geq 1+2/N$.

Since the situation is more complicated than the classical case $\al=1$, more careful calculations are often required (see Theorem \ref{strongsolution} for example).
\subsection{Function spaces and definition of solutions}\label{sec:FS}
In order to state our results, we introduce some function spaces.
For any normed space $X$, we denote by $\norm{f\mid X}$ the norm of $f$ in $X$.
\begin{defn}
  {\bf (Morrey spaces)}
  \begin{enumerate}
    \item Let $1\leq q\leq p<\infty$. The Morrey space $\cM^{p}_{q}(\R^N)$ is defined as the set of Lebesgue measurable functions $u:\R^{N}\to\R$ such that
    \begin{align*}
      \norm{u\mid\cM^{p}_{q}}\coloneqq
      \sup_{x_{0}\in\R^N}\sup_{R>0}R^{\frac{N}{p}-\frac{N}{q}}
      \onorm{u}{L^q(B(x_{0},R))}
      <\infty.
    \end{align*}
    \item Let $1\leq q\leq p<\infty$. The local Morrey space $M^{p}_{q}(\R^N)$ is defined as the set of Lebesgue measurable functions $u:\R^{N}\to\R$ such that
    \begin{align*}
      \norm{u\mid M^{p}_{q}}\coloneqq
      \sup_{x_{0}\in\R^N}\sup_{0<R\leq 1}R^{\frac{N}{p}-\frac{N}{q}}
      \onorm{u}{L^q(B(x_{0},R))}
      <\infty.
    \end{align*}
  \end{enumerate}
\end{defn}
Here and subsequently, $B(x,r)$ denotes the open ball in $\R^{N}$ with center $x\in\R^N$ and radius $r>0$.
Next, we define the space of measures of the Morrey type.
\begin{defn}{\bf (measure spaces of Morrey type)}
  \begin{enumerate}
    \item Let $1\leq p<\infty$. The measure space of the Morrey type $\cM^p(\R^N)$ is defined as the set of Radon measures $\mu$ on $\R^N$ such that
    \begin{align*}
      \onorm{\mu}{\cM^{p}}\coloneqq
      \sup_{x_{0}\in\R^N}\sup_{R>0}R^{\frac{N}{p}-N}
      \abs{\mu}(B(x_{0},R))
      <\infty.
    \end{align*}
    \item Let $1\leq p<\infty$. The local measure space of the Morrey type $M^p(\R^N)$ is defined as the set of Radon measures $\mu$ on $\R^N$ such that
    \begin{align*}
      \onorm{\mu}{M^{p}}\coloneqq
      \sup_{x_{0}\in\R^N}\sup_{0<R\leq 1}R^{\frac{N}{p}-N}
      \abs{\mu}(B(x_{0},R))
      <\infty.
    \end{align*}
    Here, $\abs{\mu}$ denotes the total variation of $\mu$.
  \end{enumerate}
\end{defn}
\begin{rmk}\label{rmkformorreynorm-1}
  The spaces $\cM^{p}_{q}(\R^N)$, $M^{p}_{q}(\R^N)$, $\cM^p(\R^N)$ and $M^p(\R^N)$ are Banach spaces with norms $\norm{\cdot\mid\cM^{p}_{q}}$, $\Mnorm{\cdot}{p}{q}$, $\onorm{\cdot}{\cM^{p}}$ and $\onorm{\cdot}{M^p}$, respectively.
\end{rmk}
We introduce the Littewood--Paley dyadic decomposition in order to define the Besov--Morrey spaces.
Let $\zeta(t)\in C_{c}^{\infty}\qty(\left[0,\infty\right[;\R)\:$ satisfy $\:0\leq\zeta(t)\leq 1,\:\zeta(t)\equiv 1$ for $t\leq\frac{3}{2}$ and $\text{supp}\:\zeta\subset\left[0,\frac{5}{3}\right[$.
We set
\begin{align*}
  \phi_{(0)}(\xi)&\coloneqq\zeta(\abs{\xi}),\\
  \phi_{0}(\xi)&\coloneqq\zeta(\abs{\xi})-\zeta(2\abs{\xi}),\\
  \phi_{j}(\xi)&\coloneqq\phi_{0}(2^{-j}\xi)\quad\text{for}\quad j\in\Z.
\end{align*}

Then, we have $\phi_{j}(\xi)\in C_{c}^{\infty}(\R^N)$ for all $j\in\Z$ and
\begin{align*}
  \sum_{j=-\infty}^{\infty}\phi_{j}(\xi)
  =
  \phi_{(0)}(\xi)+\sum_{j=1}^{\infty}\phi_{j}(\xi)=\lim_{n\rightarrow\infty}\zeta(2^{-n}\abs{\xi})=1\quad\text{for all }\xi\in\R^{N}\setminus\{0\}.
\end{align*}


\begin{defn}
  {\bf (inhomogeneous Besov--Morrey spaces)}
  Let $1\leq q\leq p<\infty,\: 1\leq r\leq\infty$ and $s\in\R$.
  The inhomogeneous Besov--Morrey space $N^{s}_{p,q,r}(\R^N)$ is defined as the set of tempered distributions $u\in\cS'(\R^N)$ such that
  \begin{align*}
    \cF^{-1}\phi_{(0)}&(\xi)\cF u,\:\cF^{-1}\phi_{j}(\xi)\cF u\in M^{p}_{q}(\R^N)\:\:\:(\text{for}\:\: j\in\Z_{\geq 1})\quad\text{and}
    \\
    \norm{u\mid N^{s}_{p,q,r}}\coloneqq
    &\norm{\cF^{-1}\phi_{(0)}(\xi)\cF u\mid M^{p}_{q}}
    +
    \onorm{\{2^{sj}\Mnorm{\cF^{-1}\phi_{j}(\xi)\cF u}{p}{q}\}_{j=1}^{\infty}}{l^r}
    <\infty.
  \end{align*}
  Here, $\cF$ denotes the Fourier transform on $\R^N$.
\end{defn}
\begin{defn}
  {\bf (homogeneous Besov--Morrey spaces)}
  Let $1\leq q\leq p<\infty,\: 1\leq r\leq\infty$ and $s\in\R$.
  The homogeneous Besov--Morrey space $\cN^{s}_{p,q,r}(\R^N)$ is defined as the set of equivalence classes of tempered distributions $u\in\cS'/\cP(\R^N)$ such that
  \begin{align*}
    \cF^{-1}\phi_{j}&(\xi)\cF u\in \cM^{p}_{q}(\R^N)\:\:\:(\text{for}\:\: j\in\Z)\quad\text{and}
    \\
    \norm{u\mid \cN^{s}_{p,q,r}}\coloneqq
    &\norm{\{2^{js}\norm{\cF^{-1}\phi_{j}(\xi)\cF u\mid\cM^{p}_{q}}\}_{j\in\Z}\mid l^r}
    <\infty.
  \end{align*}
  Here, $\cP(\R^N)$ denotes the set of polynomials with $N$ variables.
\end{defn}
Note that we can regard an element of homogeneous Besov--Morrey spaces as a tempered distribution if $s<\frac{N}{p}$ or $(s,r)=(N/p,1)$ (see Proposition 2.10 of \cite{KozonoYamazaki}).
Also, we have the inclusion relation $\cN^{s}_{p,q,r}\subset N^{s}_{p,q,r}$ if $s<0$ (see Proposition 2.14 of \cite{KozonoYamazaki}).
Next, by using the heat kernel and the Wright type function, we give the definition of the solution to problem \eqref{eq}.
For every $t>0$ and for every $u\in\cS'(\R^N)$, we put
\begin{align*}
  e^{t\Del}u\coloneqq\cF^{-1}e^{-t\abs{\xi}^2}\cF u.
\end{align*}
\begin{defn}\label{def:timefracheatkernel}
  Let $1\leq q\leq p<\infty,\: 1\leq r\leq\infty$ and $s\in\R$.
  For $\mu\in N^{s}_{p,q,r}$, we set
  \begin{align*}
    P_{\al}(t)\mu(x)&\coloneqq\int_{0}^{\infty}\Phi_{\al}(\theta)e^{t^{\al}\theta\Del}\mu(x)\,d\theta\:\quad(t>0)
    \\
    S_{\al}(t)\mu(x)&\coloneqq\al\int_{0}^{\infty}\theta\Phi_{\al}(\theta)e^{t^{\al}\theta\Del}\mu(x)\,d\theta\quad(t>0)
  \end{align*}
  Here, $\Phi_{\al}(x)$ denotes the Wright type functions given by
  \begin{align*}
    \Phi_{\al}(z)\coloneqq\sum_{k=0}^{\infty}\frac{(-z)^{k}}{k!\Gam(-\al k+1-\al)},
    \quad 0<\al<1, z\in\C.
  \end{align*}
\end{defn}

\begin{defn}\label{def:inteq}
  Let $1\leq q\leq p<\infty,\: 1\leq r\leq\infty,\:s\in\R$ and $T>0$.
  For $\mu\in N^{s}_{p,q,r}$, we say that a function $u(x,t)$ is a mild solution to problem \eqref{eq} on $\R^{N}\times\left]0,T\right[$ if it satisfies the following integral equation
    \begin{align}\label{eq:int}
      u(t)=P_{\al}(t)\mu+
      \int_{0}^{t}(t-\tau)^{\al -1}S_{\al}(t-\tau)|u|^{\gam -1}u(\tau)\,d\tau
    \end{align}
  for $a.e.\:(x,t)\in\R^{N}\times\left]0,T\right[$.
\end{defn}
This definition of a mild solution is based on \cite{GWBook} (see also \cite{Bazhlekova,ZhangSun}).
\subsection{Main results}
Now, we are ready to state our main theorems.
\begin{thm}\label{maintheorem}
  Let $\gam>1,\:\gam\leq q\leq p<\infty,\:\max\{-2/\al\gam, -2\}<s<0,$ and
   $s\geq\frac{N}{p}-\frac{2}{\gam -1}$.
   Then, there exist $\del >0$ and $M>0$ such that for every $\mu\in N^{s}_{p,q,\infty}$ satisfying
   \begin{align}\label{cond:highfreq}
     \limsup_{j\rightarrow\infty}2^{sj}\Mnorm{\cF^{-1}\phi_{j}\cF\mu}{p}{q}<\del,
   \end{align}
   problem \eqref{eq} possesses the unique solution $u$ in $\R^{N}\times\left]0,T\right[$ for some $T>0$ with $\sup_{0<t\leq T}t^{-s\al/2}\Mnorm{u(\cdot,t)}{p}{q}\leq M$.
\end{thm}
\begin{rmk}\label{rmk:optimality}
  If $\gam\leq 1+2/N$, by using Remark \ref{rmkformorreynorm}, Proposition \ref{inclusion} and Proposition \ref{Sobolevtypeineq}, we see that
  \begin{align}\label{eq:doublecrit}
    L^{1}(\R^{N})
    \subset
    M^{1}_{1}(\R^N)
    \subset
    N^{0}_{1,1,\infty}(\R^N)
    \subset
    N^{-N+N/\gam}_{\gam,\gam,\infty}(\R^N).
  \end{align}
  The space appearing in the last of \eqref{eq:doublecrit} satisfies the assumptions of Theorem \ref{maintheorem}.
  If $\gam>1+2/N$, we can take $p$ such that
  $\max\{ \gam, q_{c} \}<p<q_{c}\gam$ and in the same way, we have
  \begin{align*}
    L^{q_{c}}(\R^{N})
    \subset
    M^{q_{c}}_{q_{c}}(\R^N)
    \subset
    N^{0}_{q_{c},q_{c},\infty}(\R^N)
    \subset
    N^{N/p-2/(\gam-1)}_{p,p,\infty}(\R^N).
  \end{align*}
\end{rmk}
\begin{rmk}\label{highfreq}
  By using Theorem \ref{maintheorem} and Remark \ref{rmk:optimality}, we can say that there exists a solution to the problem \eqref{eq} even in the so-called doubly critical case ($\gam=1+2/N$ and initial data $\mu\in L^{1}$). Recall that problem $\eqref{eq}$ has the following scale invariant property:
  If a function $u(x,t)$ is a solution to problem $\eqref{eq}$ with initial data $\mu(x)$, then a function $u_{\lam}(x,t)\coloneqq\lam^{2\al/(\gam-1)}u(\lam^{\al}x,\lam^{2}t)$ is also a solution to problem $\eqref{eq}$ with initial data $\mu_{\lam}(x)\coloneqq\lam^{2\al/(\gam-1)}\mu(\lam^{\al}x)$.
  For any $L^1$ function $\mu(x)$, we see that $\mu_{\lam}$ satisfies the condition
  \[
  \limsup_{j\to\infty}2^{(-N+N/\gam)j}
  \norm{\cF^{-1}\phi_{j}\cF\mu_{\lam}\mid M^{\gam}_{\gam}}
  <\del
  \]
  by taking $\lam>0$ so small. To see this, we use the embedding \eqref{eq:doublecrit} to get
  \begin{align*}
    \limsup_{j\to\infty}2^{(-N+N/\gam)j}
    \norm{\cF^{-1}\phi_{j}\cF\mu_{\lam}\mid M^{\gam}_{\gam}}
    &\leq
    \norm{\mu_{\lam}\mid N^{-N+N/\gam}_{\gam,\gam,\infty}}
    \leq
    C\norm{\mu_{\lam}\mid M^{1}_{1}}
    \\
    &=
    C\sup_{x\in\R^{N}}
    \int_{|x-y|\leq\lam^{\al}}|\mu(y)|\,dy
  \end{align*}
  for all $\lam>0$. By absolute continuity of $\mu$, the last term goes to zero as $\lam\rightarrow 0$.
\end{rmk}
\begin{rmk}
  The following distributions belong to Besov--Morrey spaces.
  \begin{enumerate}
    \item When $\gam> p_{F}=1+2/N$, we can take $p$ such that
    $\max\{ \gam, q_{c} \}<p<q_{c}\gam$ with $q_{c}=\frac{N(\gam-1)}{2}$. By using Proposition \ref{inclusion} and Proposition \ref{Sobolevtypeineq}, we have
    \begin{align*}
      \abs{x}^{-\frac{2}{\gam-1}}
      \in
      M^{q_{c}}_{q_{c}\gam/p}
      \subset
      N^{0}_{q_{c},q_{c}\gam/p,\infty}
      \subset
      N^{N/p-2/(\gam-1)}_{p,\gam,\infty}.
    \end{align*}
    Since the assumption of Theorem \ref{maintheorem} is satisfied with $N^{N/p-2/(\gam-1)}_{p,\gam,\infty}$ for above $p$, we see by \eqref{cond:highfreq} that there exists $c>0$ such that, if an initial datum $\mu$ satisfies
    \[
    0\leq\mu(x)\leq c|x|^{-\frac{2}{\gam-1}},\quad x\in\R^{N},
    \]
    then problem $\eqref{eq}$ possesses a local-in-time solution.
    Note that in the case of $\al=1$, problem \eqref{eq} possesses a local-in-time solution when $c<<1$ and does not possess when $c>>1$, as demonstrated in \cite{HisaIshige2018}.
    \item Let $1<\gam<1+2/N$. We see that the Dirac delta measure $\del(x)$ satisfies
    \[
    \del(x) \in M^1 \subset N^{0}_{1,1,\infty}
    \subset N^{-N+N/\gam}_{\gam,\gam,\infty}
    \]
    by using Proposition \ref{inclusion} and Proposition \ref{Sobolevtypeineq}.
    Furthermore, since
    \[
    \pl_{x_{i}}\del(x)\in N^{-N+N/\gam-1}_{\gam,\gam,\infty}
    \quad\text{for}\quad
    i\in\{1,...,N\},
    \]
    we can take $\pl_{x_{i}}\del(x)$ as initial data of problem \eqref{eq}
    for $\gam$ satisfiying
    \begin{align*}
      \begin{cases}
        \gam<\frac{N+2/\al}{N+1} \quad\text{if}\quad N=1,2,
        \\
        \gam<\min\left\{\frac{N}{N-1},\frac{N+2/\al}{N+1}\right\} \quad\text{if}\quad N\geq 3.
      \end{cases}
    \end{align*}
    To our best knowledge, this type of initial data is not yet considered for the problem \eqref{eq}.
  \end{enumerate}
\end{rmk}
\begin{rmk}
  In our definition of mild solution, we do not think about whether the solution converges to the initial value or what class the solution belongs to.
  If we make an additional assumption, we can prove the convergence (see Theorem \ref{thm:convtoini}).
  For the latter, we prove that $u\in C\qty(\:\left]0,T\right];N^{s}_{p,q,\infty}(\R^N))$ (see Theorem \ref{strongsolution}).
\end{rmk}
In the following, we describe our global-in-time existence theorem.
\begin{thm}\label{globalsol}
  Let $\gam$ satisfy the condition
  \begin{align*}
    \gam
    >
    \gam(\al)\coloneqq
    \max\qty{1+\frac{2\al}{N\al+2(1-\al)}
    ,\:\:\frac{4-N+\sqrt{N^{2}+16}}{4}}
  \end{align*}
  and take $\gam\leq q\leq p<\infty$ satisfying
  \[
  \frac{N(\gam-1)}{2}< p
  <\min\qty{\frac{N(\gam-1)}{(4-2\gam)_{+}},
  \frac{N\gam\al(\gam-1)}{2(1+\al\gam-\gam)_{+}}}.
  \]
  Then, there exist $\del >0$ and $M>0$ such that for every $\mu\in \cN^{N/p-2/(\gam -1)}_{p,q,\infty}$ satisfying $\norm{\mu\mid\cN^{N/p-2/(\gam -1)}_{p,q,\infty}}<\del$,
  problem \eqref{eq} possesses the unique solution $u$ in $\R^{N}\times\left]0,\infty\right[$ with $\sup_{0<t}t^{\al/(\gam -1)-N\al/2p}\norm{u(\cdot,t)\mid\cM^{p}_{q}}\leq M$.
\end{thm}
\begin{rmk}\label{rmk:globalsol}
  The assumptions on $p$ and $q$ in Theorem \ref{globalsol} arise from Theorem \ref{maintheorem} with $s=N/p-2/(\gam -1)$. The condition $\gam>\gam(\al)$ is necessary in order to take such $p$ and $q$.

\end{rmk}
\begin{rmk}
  Zhang and Sun \cite{ZhangSun} determined a Fujita exponent for problem \eqref{eq} with $\mu\in C_{0}(\R^{N})\coloneqq\{ u\in C(\R^{N})\mid \lim_{\abs{x}\to\infty}u(x)=0 \}$;
  \begin{enumerate}[{\rm(i)}]
    \item For $\mu\geq 0$ and $\mu\not\equiv 0$, if $1<\gam<1+2/N$, then the solution of problem \eqref{eq} blows up in finite time.
    \item\label{ZS:2} For $\mu\in L^{q_{c}}(\R^N)$ with $q_{c}=N(\gam-1)/2$, if $\gam\geq 1+2/N$ and $\norm{\mu\mid L^{q_{c}}}$ is sufficiently small, then problem \eqref{eq} has a time global solution.
  \end{enumerate}
  On the other hand, the condition of $\gam$ in Theorem \ref{globalsol} is weaker than the usual Fujita supercritilcal condition $\gam>1+2/N$.
  Note that our setting allows spatial unboundedness of solutions for $t>0$ and in general, for singular initial data, the solution of \eqref{eq} is not necessarily bounded for $t>0$ (see Remark \ref{rmkPalpha}), and thus our results are not in conflict with those of \cite{ZhangSun}.

\end{rmk}
\begin{rmk}
  The exponent
  \[
  \gam(\al)=
  \max\qty{1+\frac{2\al}{N\al+2(1-\al)}
  ,\:\:\frac{4-N+\sqrt{N^{2}+16}}{4}}
  \]
   partially corresponds to the exponent that appears in Remark 1 of \cite{KLT}.
   They showed that there is no nontrivial global weak solution of problem \eqref{eq} in the case of
   $
     1<\gam\leq 1+2\al/(N\al+2(1-\al))
   $.
\end{rmk}


The rest of this paper is organized as follows: In Section 2, we establish the estimates for operators $P_{\al}(t)$ and $S_{\al}(t)$ in Besov--Morrey spaces by use of the Wright type function. In Section 3, we get the time local well-posedness for singular initial data (Theorem  \ref{maintheorem} and \ref{thm:convtoini}). In Section 4, we construct a global-in-time solution for sufficiently small initial data in the sense of ``scale-invariant'' Besov--Morrey norm (Theorem \ref{globalsol}).
In what follows, the letter $C$ denotes generic positive constants and they may change from line to line.
%
\section{Preliminaries}

We recall some preliminary facts about Morrey spaces and Besov--Morrey spaces.
\begin{prop}\label{rmkformorreynorm}
  {\bf (properties of Morrey spaces)}
  \begin{enumerate}
    \item\label{rmkformorreynorm-2}
      $\Mnorm{u^{\gam}}{p/\gam}{q/\gam}=\Mnorm{u}{p}{q}^{\gam}$ for all $u\in M^{p}_{q}$, $\max\qty{1,\gam}\leq q\leq p.$
    \item\label{rmkformorreynorm-3}
    if $u\in M^{p}_{q}$, then $\abs{u(x)}^{q}dx\in M^{\frac{p}{q}}$ and $\onorm{\abs{u(x)}^{q}dx}{M^{\frac{p}{q}}}=\Mnorm{u}{p}{q}^q$.
    \item $\cM^{p}_{q}(\R^N)\subset M^{p}_{q}(\R^N).$
    \item\label{rmkformorreynorm-4}
    $L^p(\R^N)\subset M^{p}_{q}(\R^N)$.
    \item $\abs{x}^{-\frac{N}{p}}\notin L^{p}(\R^N)$ but if $q<p$,
    $\abs{x}^{-\frac{N}{p}}\in M^{p}_{q}(\R^N)$.
  \end{enumerate}
\end{prop}
We can find in \cite{KozonoYamazaki} the following two Propositions (Proposition 2.11 and 2.5).
\begin{prop}\label{inclusion}
  For $1\leq q\leq p<\infty$, we have the inclusion relations
  \begin{align*}
    N_{p,q,1}^{0}\subset M^{p}_{q}\subset N_{p,q,\infty}^{0},
    \quad M^{p}\subset N^{0}_{p,1,\infty}.
  \end{align*}
\end{prop}
\begin{prop}\label{Sobolevtypeineq}
  {\bf (Sobolev embedding)}
  Let $1\leq q\leq p<\infty$, $r\in[1,\infty]$ and $s\in\R$.
  Then, for all $l\in\left]0,1\right[$,
  \begin{align*}
    N^{s}_{p,q,r}\subset N^{s-N(1-l)/p}_{p/l,q/l,r},
    \quad N^{s}_{p,q,r}\subset B^{s-N/p}_{\infty, r}.
  \end{align*}
  Here, $B^{s-N/p}_{\infty,r}$ denotes inhomogeneous Besov spaces defined in the same way as inhomogeneous Besov--Morrey spaces by replacing Morrey spaces with standard $L^{\infty}$ spaces.
\end{prop}
The Wright type function $\Phi_{\al}(x)\:(0<\al<1)$ has the following properties (see \cite{GLM}).
\begin{lem}\label{wright}
  {\bf (properties of Wright type function)}
  \begin{enumerate}[\rm(a)]
    \item\label{wright-a}
    $\Phi_{\al}(\theta)\geq 0$ for $\theta\geq 0$ and $\int_{0}^{\infty}\Phi_{\al}(\theta)\,d\theta =1$.
    \item\label{wright-b}
    $\int_{0}^{\infty}\theta^{r}\Phi_{\al}(\theta)\,d\theta =\frac{\Gam(1+r)}{\Gam(1+\al r)}$ for $r>-1$.
  \end{enumerate}
\end{lem}
The operator $P_{\al}(t)$ defined in Definition \ref{def:timefracheatkernel} is the solution operator of the linear problem which is associated with problem \eqref{eq} (see Theorem 3.2 of \cite{Bazhlekova}).
That is, $P_{\al}(t)\mu$ satisfies
\begin{align*}
  ^{C}\pl_{t}^{\al}P_{\al}(t)\mu=\Del P_{\al}(t)\mu.
\end{align*}
for suitable smooth function $\mu$.
\begin{rmk}\label{rmkPalpha}
  In contrast to the case of $\al=1$, we generally cannot expect that solutions of time fractional heat equations are spatially bounded for $t>0$.
  For instance, for Dirac delta measure $\del$, we have
  \begin{align*}
    (P_{\al}(t)\del)(0)
    &=
    \left.\int_{0}^{\infty}
    \Phi_{\al}(\theta)(e^{t^{\al}\theta\Del}\del)(x)
    \,d\theta\right|_{x=0}
    \\
    &=
    \left.\int_{0}^{\infty}
    \frac{\Phi_{\al}(\theta)}{\sqrt{4\pi t^{\al}\theta}^{N}}
    \int_{\R^N}
    \exp\qty(\frac{\abs{x-y}^{2}}{4t^{\al}\theta})\del(y)
    \,dy
    d\theta\right|_{x=0}
    \\
    &=
    \left.\int_{0}^{\infty}
    \frac{\Phi_{\al}(\theta)}{\sqrt{4\pi t^{\al}\theta}^{N}}
    \exp\qty(\frac{\abs{x}^{2}}{4t^{\al}\theta})
    \,d\theta\right|_{x=0}
    =
    \frac{1}{\sqrt{4\pi t^{\al}}^{N}}
    \int_{0}^{\infty}
    \theta^{-N/2}\Phi_{\al}(\theta)
    \,d\theta,
  \end{align*}
  but the last integral has finite value only if $N<2$ from Lemma \ref{wright}-\eqref{wright-b}.
  See also \cite{EK}.
\end{rmk}
The following property is proved in \cite{KozonoYamazaki}(Theorem 3.1).
\begin{lem}\label{smoothingforGauss}
  Let $1\leq q\leq p<\infty$, $r\in[1,\infty]$ and $s\leq\sigma $.
  Then, there exists $C>0$ such that for all $u\in N^{s}_{p,q,r}$, the estimate
  \begin{align*}
    \BMnorm{e^{t\Del}u}{\sigma}{p}{q}{r}
    \leq C
    \qty(1+t^{\frac{s-\sigma}{2}})
    \BMnorm{u}{s}{p}{q}{r}
    \quad\text{for}\quad t>0
  \end{align*}
  holds. Furthermore, if $s<\sigma$, the estimate
  \begin{align*}
    \BMnorm{e^{t\Del}u}{\sigma}{p}{q}{1}
    \leq C
    \qty(1+t^{\frac{s-\sigma}{2}})
    \BMnorm{u}{s}{p}{q}{\infty}
    \quad\text{for}\quad t>0
  \end{align*}
  holds.
\end{lem}
From this lemma, we can obtain estimates for the operators $S_{\al}(t)$ and $P_{\al}(t)$.
\begin{lem}\label{smoothingforSalpha}
  Let $1\leq q\leq p<\infty,\: r\in[1,\infty],\: s\leq\sigma $ and $4>\sigma -s$.
  Then, there exists $C>0$ such that for all $u\in N^{s}_{p,q,r}$, the estimate
  \begin{align*}
    \BMnorm{S_{\al}(t)u}{\sigma}{p}{q}{r}
    \leq C
    \qty(1+t^{\frac{(s-\sigma)\al}{2}})
    \BMnorm{u}{s}{p}{q}{r}
    \quad\text{for}\quad t>0
  \end{align*}
  holds. Furthermore, if $s<\sigma$, the estimate
  \begin{align*}
    \BMnorm{S_{\al}(t)u}{\sigma}{p}{q}{1}
    \leq C
    \qty(1+t^{\frac{(s-\sigma)\al}{2}})
    \BMnorm{u}{s}{p}{q}{\infty}
    \quad\text{for}\quad t>0
  \end{align*}
  holds.
\end{lem}
\begin{proof}
  Lemma \ref{smoothingforGauss}, Lemma \ref{wright}-\eqref{wright-b} and Minkowski inequality give
  \begin{align*}
    \BMnorm{S_{\al}(t)u}{\sigma}{p}{q}{r}
    &=\BMnorm{\al\int_{0}^{\infty}\theta\Phi_{\al}(\theta)e^{t^{\al}\theta\Del}u(x)\,d\theta}{\sigma}{p}{q}{r}
    \\
    &\leq
    \al\int_{0}^{\infty}\theta\Phi_{\al}(\theta)\BMnorm{e^{t^{\al}\theta\Del}u}{\sigma}{p}{q}{r}\,d\theta
    \\
    &\leq
    C\al\int_{0}^{\infty}
    \theta\Phi_{\al}(\theta)\qty(1+(t^{\al}\theta)^{\frac{s-\sigma}{2}})
    \BMnorm{u}{s}{p}{q}{r}\,d\theta\\
    &=
    C\BMnorm{u}{s}{p}{q}{r}
    \qty(
    \int_{0}^{\infty}\theta\Phi_{\al}(\theta)\,d\theta
    \:+\:
    t^{\frac{(s-\sigma)\al}{2}}
    \int_{0}^{\infty}
    \theta^{1+\frac{s-\sigma}{2}}
    \Phi_{\al}(\theta)\,d\theta
    )
    \\
    &\leq
    C\BMnorm{u}{s}{p}{q}{r}
    \qty(1+t^{\frac{(s-\sigma)\al}{2}}).
  \end{align*}
  Note that $\:\:1+\frac{s-\sigma}{2}>-1\:\:$ if $\:\:4>\sigma-s$.
\end{proof}
\begin{lem}\label{smoothingforPalpha}
  Let $1\leq q\leq p<\infty,\: r\in[1,\infty],\: s\leq\sigma $ and $2>\sigma -s$.
  Then, there exists $C>0$ such that for all $u\in N^{s}_{p,q,r}$, the estimate
  \begin{align*}
    \BMnorm{P_{\al}(t)u}{\sigma}{p}{q}{r}
    \leq C
    \qty(1+t^{\frac{(s-\sigma)\al}{2}})
    \BMnorm{u}{s}{p}{q}{r}
    \quad\text{for}\quad t>0
  \end{align*}
  holds. Furthermore, if $s<\sigma$, the estimate
  \begin{align*}
    \BMnorm{P_{\al}(t)u}{\sigma}{p}{q}{1}
    \leq C
    \qty(1+t^{\frac{(s-\sigma)\al}{2}})
    \BMnorm{u}{s}{p}{q}{\infty}
    \quad\text{for}\quad t>0
  \end{align*}
  holds.
\end{lem}
\begin{proof}
  The proof is similar to that of Lemma \ref{smoothingforSalpha}.
\end{proof}
\section{Existence results}
This section is divided into two subsections.
In the first subsection, we construct the local-in-time mild solution to problem \eqref{eq} by using an iterative method.
In the second subsection, we consider the convergence of solutions to initial data and the continuity of solutions in time.
\subsection{Local-in-time mild solution}\label{localsol}

We begin the proof of Theorem \ref{maintheorem}.
Fix $T>0.$ We write
\begin{align*}
  &X_{T}\coloneqq
  \qty{
  u(x,t) : \text{Lebesgue measurable in } \R^{N}\times\left]0,T\right[
  \mid
  \xtnorm{u}<\infty
  },
\end{align*}
where
\begin{align*}
  \norm{u\mid X_{T}}\coloneqq\sup_{0<t<T}t^{-s\al/2}\Mnorm{u(\cdot,t)}{p}{q}.
\end{align*}
We set $u_{0}\coloneqq P_{\al}(t)\mu$ and define $u_{n}\:(n\in\Z_{\geq 1})$ inductively by
\begin{align*}
  u_{n}(t)\coloneqq u_{0}(t)+\int_{0}^{t}(t-\tau)^{\al -1}S_{\al}(t-\tau)|u_{n-1}|^{\gam -1}u_{n-1}(\tau)\,d\tau.
\end{align*}
We prepare the following three lemmata to prove Theorem \ref{maintheorem}.
\begin{lem}\label{u0bound}
  Let $1\leq q\leq p<\infty, -2<s<0$ and $\del>0$.
  Then, there exists $C_{2}>0$ such that for all $\mu(x)\in N^{s}_{p,q,\infty}$ satisfying
  $$
    \limsup_{j\rightarrow\infty}2^{sj}\Mnorm{\cF^{-1}\phi_{j}\cF\mu}{p}{q}<\del,
  $$ there exists a positive number $T$ such that
  \[
  \norm{u_{0}\mid X_{T}}\leq C_{2}\del.
  \]
\end{lem}
\begin{proof}
  Let $C_{S}$ be a number appering in Lemma \ref{smoothingforPalpha}, namely:
  \begin{align*}
    \BMnorm{P_{\al}(t)u}{0}{p}{q}{1}
    \leq C_{S}
    \qty(1+t^{s\al/2})
    \BMnorm{u}{s}{p}{q}{\infty}
    \quad\text{for}\quad t>0, \: u\in N^{s}_{p,q,\infty}.
  \end{align*}
  Set $C_{1}\coloneqq\max\{ 1, 2\onorm{\cF^{-1}\phi_{0}}{L^{1}(\R^N)}\}$ and
   $C_{0}\coloneqq C_{S}C_{1}$.
  There exist $\ep>0$ and $m\in\N$ such that for every $j\geq m$, $2^{sj}\Mnorm{\cF^{-1}\phi_{j}\cF\mu}{p}{q}\leq\del-\ep<\del$.
  We put $\mu_{1}\coloneqq\cF^{-1}\phi_{(0)}(2^{-m}\:\cdot\:)\cF\mu$ and $\mu_{2}\coloneqq\mu-\mu_{1}$. By a direct calculation, we can check the following:
  \begin{align*}
    \cF^{-1}\phi_{j}\cF\mu_{1}
    =
    \begin{cases}
      \cF^{-1}\phi_{j}\cF\mu & \text{for}\quad j\leq m-1,\\
      \cF^{-1}(\phi_{m-1}+\phi_{m})\phi_{j}\cF\mu & \text{for}\quad j=m,m+1,\\
      0 & \text{for}\quad j\geq m+2,
    \end{cases}
  \end{align*}
  and
  \begin{align*}
    \cF^{-1}\phi_{j}\cF\mu_{2}
    =
    \begin{cases}
      0 & \text{for}\quad j\leq m-1,\\
      \cF^{-1}(\phi_{m+1}+\phi_{m+2})\phi_{j}\cF\mu & \text{for}\quad j=m,m+1,\\
      \cF^{-1}\phi_{j}\cF\mu & \text{for}\quad j\geq m+2.
    \end{cases}
  \end{align*}
  This yields $\BMnorm{\mu_{2}}{s}{p}{q}{\infty}\leq C_{1}(\del-\ep)$.
  Lemma \ref{smoothingforPalpha} gives that
  \begin{align*}
    t^{-s\al/2}
    \BMnorm{P_{\al}(t)\mu_{2}}{0}{p}{q}{1}
    &\leq C_{S}
    \qty(1+t^{s\al/2})
    t^{-s\al/2}
    \BMnorm{\mu_{2}}{s}{p}{q}{\infty}
    \\
    \leq
    C_{S}C_{1}(\del-\ep)
    \qty(t^{-s\al/2}+1)
    &\leq
    C_{0}(\del-\ep)
    \qty(T^{-s\al/2}+1)
    <
    C_{0}\del-\frac{\ep C_{0}}{2}
  \end{align*}
  for all $t\in\left]0,T\right[$, by taking $T>0$ sufficiently small.
  On the other hand, $\mu_{1}\in N_{p,q,\infty}^{s/2}$ and $-s/2<1<2$. Thus, again using Lemma \ref{smoothingforPalpha},
  \begin{align*}
    t^{-s\al/2}
    \BMnorm{P_{\al}(t)\mu_{1}}{0}{p}{q}{1}
    &\leq
    t^{-s\al/2}
    C_{S}\qty(1+t^{s\al/4})
    \BMnorm{\mu_{1}}{s/2}{p}{q}{\infty}\\
    &\leq
    C_{S}T^{-s\al/4}
    \qty(1+T^{-s\al/4})
    \BMnorm{\mu_{1}}{s/2}{p}{q}{\infty}
    <
    \frac{C_{0}\ep}{2}
  \end{align*}
  for all $t\in\left]0,T\right[$, by taking $T>0$ sufficiently small.
  Therefore,
  \begin{align*}
    t^{-s\al/2}
    \Mnorm{P_{\al}(t)\mu}{p}{q}
    &\leq C
    t^{-s\al/2}
    \BMnorm{P_{\al}(t)\mu}{0}{p}{q}{1}\\
    &\leq
    Ct^{-s\al/2}
    \BMnorm{P_{\al}(t)\mu_{1}}{0}{p}{q}{1}
    +
    Ct^{-s\al/2}
    \BMnorm{P_{\al}(t)\mu_{2}}{0}{p}{q}{1}
    \\
    &\leq
    C\qty(C_{0}\del-\frac{\ep C_{0}}{2}+\frac{C_{0}\ep}{2})
    =CC_{0}\del=C\del
  \end{align*}
  for all $t\in\left]0,T\right[$, by taking $T>0$ sufficiently small. Thus, we reach the conclusion.
\end{proof}
\begin{lem}\label{zenkashiki}
  Let $\gam>1,\:T\leq 1,\:\gam\leq q\leq p<\infty,\:-2/\al\gam<s<0,$ and

   $s\geq\frac{N}{p}-\frac{2}{\gam -1}$.
  Then, there exists $C_{3}>0$ independent of $T$ such that
  \begin{align*}
    \norm{u_{n+1}\mid X_{T}}
    &\leq
    \norm{u_{0}\mid X_{T}}+
    C_{3}\norm{u_{n}\mid X_{T}}^{\gam}
  \end{align*}
  for $n=0,1,...$.
\end{lem}
\begin{proof}
  By using Proposition \ref{inclusion}, Proposition \ref{Sobolevtypeineq},
  Lemma \ref{smoothingforSalpha} and Proposition \ref{rmkformorreynorm}.\ref{rmkformorreynorm-2}, we obtain the estimate
  \begin{align*}
    \Mnorm{u_{n+1}(\cdot,t)-u_{0}(\cdot,t)}{p}{q}
    \leq C&\BMnorm{u_{n+1}(\cdot,t)-u_{0}(\cdot,t)}{0}{p}{q}{1}
    \\
    \leq C&
    \int_{0}^{t}(t-\tau)^{\al -1}\BMnorm{S_{\al}(t-\tau)|u_{n}|^{\gam-1}u_{n}(\cdot,\tau)}{0}{p}{q}{1}\,d\tau
    \\
    \leq C&
    \int_{0}^{t}(t-\tau)^{\al -1}
    \BMnorm{S_{\al}(t-\tau)|u_{n}|^{\gam-1}u_{n}(\cdot,\tau)}{\frac{N(\gam -1)}{p}}{p/\gam}{q/\gam}{1}\,d\tau
    \\
    \leq C&
    \int_{0}^{t}(t-\tau)^{\al -1}
    \qty(1+(t-\tau)^{-\frac{N(\gam -1)\al}{2p}})
    \\
    &\qquad\qquad\qquad\qquad\times
    \BMnorm{|u_{n}|^{\gam-1}u_{n}(\cdot,\tau)}{0}{p/\gam}{q/\gam}{\infty}
    \,d\tau
    \\
    \leq C&
    \int_{0}^{t}(t-\tau)^{\al -1-\frac{N(\gam -1)\al}{2p}}
    \Mnorm{|u_{n}|^{\gam-1}u_{n}(\cdot,\tau)}{p/\gam}{q/\gam}
    \,d\tau
    \\
    \leq C&
    \norm{u_{n}\mid X_{T}}^{\gam}\int_{0}^{t}(t-\tau)^{\al -1-\frac{N(\gam -1)\al}{2p}}
    \tau^{\frac{s\gam\al}{2}}
    \,d\tau.
  \end{align*}
  Here, we used the relation $N(\gam-1)/2p<1$ which comes from $s<0$ and $s\geq N/p-2/(\gam-1)$.
  Therefore, we have
  \begin{align*}
    t^{-\frac{s\al}{2}}\norm{u_{n+1}-u_{0}\mid M^{p}_{q}}
    &\leq C
    \norm{u_{n}\mid X_{T}}^{\gam}
    t^{\al -\frac{N(\gam -1)\al}{2p}+\frac{s\gam\al}{2}-\frac{s\al}{2}}
    \\
    &\leq C
    \norm{u_{n}\mid X_{T}}^{\gam}
  \end{align*}
  for all $t<T(\leq 1)$. Then, we take a supremum over $t\in\left] 0,T \right[$ and we get
  \begin{align*}
  \norm{u_{n+1}-u_{0}\mid X_{T}}
  \leq C
  \norm{u_{n}\mid X_{T}}^{\gam}
\end{align*}
for all $T(\leq 1)$.
\end{proof}
By using Lemma \ref{u0bound} and Lemma \ref{zenkashiki} repeatedly, we can check that functions $u_{n}$ have a bound
\begin{align}\label{bound}
\sup_{n\in\N}\norm{u_{n}\mid X_{T}}\leq M:=2C_{2}\del
\end{align}
if $\del>0$ satisfies $2^{\gam}C_{2}^{\gam-1}C_{3}\del^{\gam-1}\leq 1$.
\begin{lem}
  Let $\gam>1,\:T\leq 1,\:\gam\leq q\leq p<\infty,\:\max\{-2,-2/\al\gam\}<s<0,$ and

   $s\geq\frac{N}{p}-\frac{2}{\gam -1}$.
  Then, there exists $C>0$ independent of $T$ such that
  \begin{align*}
    \norm{u_{n+2}-u_{n+1}\mid X_{T}}\leq C\del^{\gam -1}
    \norm{u_{n+1}-u_{n}\mid X_{T}}
  \end{align*}
  for $n=0,1,...$.
\end{lem}
\begin{proof}
By using Proposition \ref{inclusion}, Proposition \ref{Sobolevtypeineq},
Lemma \ref{smoothingforSalpha} and Proposition \ref{rmkformorreynorm}.\ref{rmkformorreynorm-2}, we see that
  \begin{align*}
    &\Mnorm{u_{n+2}(\cdot,t)-u_{n+1}(\cdot,t)}{p}{q}
    \\
    \leq C&
    \int_{0}^{t}(t-\tau)^{\al -1}
    \BMnorm{
    S_{\al}(t-\tau)\qty(
    |u_{n+1}|^{\gam-1}u_{n+1}(\cdot,\tau)-|u_{n}|^{\gam-1}u_{n}(\cdot,\tau))
    }{0}{p}{q}{1}\,d\tau
    \\
    \leq C&
    \int_{0}^{t}
    (t-\tau)^{\al -1-\frac{N(\gam -1)\al}{2p}}
    \BMnorm{|u_{n+1}|^{\gam-1}u_{n+1}(\cdot,\tau)-|u_{n}|^{\gam-1}u_{n}(\cdot,\tau)}{0}{p/\gam}{q/\gam}{\infty}
    \,d\tau
    \\
    \leq C& \int_{0}^{t}(t-\tau)^{\al -1-\frac{N(\gam -1)\al}{2p}}
    \Mnorm{|u_{n+1}|^{\gam-1}u_{n+1}(\cdot,\tau)-|u_{n}|^{\gam-1}u_{n}(\cdot,\tau)}{p/\gam}{q/\gam}
    \,d\tau
    \\
    \leq C& \int_{0}^{t}(t-\tau)^{\al -1-\frac{N(\gam -1)\al}{2p}}
    \Mnorm{u_{n+1}(\cdot,\tau)-u_{n}(\cdot,\tau)}{p}{q}\\
    &\qquad\qquad\qquad\qquad\times\qty(
    \Mnorm{u_{n+1}(\cdot,\tau)}{p}{q}^{\gam -1}
    +\Mnorm{u_{n}(\cdot,\tau)}{p}{q}^{\gam -1}
    )
    \,d\tau\\
    \leq C& \int_{0}^{t}(t-\tau)^{\al -1-\frac{N(\gam -1)\al}{2p}}\tau^{\frac
    {s\al\gam}{2}}
    \,d\tau
    \norm{u_{n+1}-u_{n}\mid X_{T}}
    \\
    &\qquad\qquad\qquad\qquad\times
    \qty(
    \norm{u_{n+1}\mid X_{T}}^{\gam -1}
    +\norm{u_{n}\mid X_{T}}^{\gam -1}
    ).
  \end{align*}
  In the fourth inequality, we used the mean value theorem.
  Therefore, we have
  \begin{align*}
    t^{-\frac{s\al}{2}}\norm{u_{n+2}-u_{n+1}\mid M^{p}_{q}}
    \leq C&
    \norm{u_{n+1}-u_{n}\mid X_{T}}
    t^{\al -\frac{N(\gam -1)\al}{2p}+\frac{s\gam\al}{2}-\frac{s\al}{2}}\\
    &\quad\quad\quad\quad\times
    \qty(
    \norm{u_{n+1}\mid X_{T}}^{\gam -1}
    +\norm{u_{n}\mid X_{T}}^{\gam -1}
    )\\
    \leq C&
    \norm{u_{n+1}-u_{n}\mid X_{T}}
    \qty(
    \norm{u_{n+1}\mid X_{T}}^{\gam -1}
    +\norm{u_{n}\mid X_{T}}^{\gam -1}
    )
  \end{align*}
  for all $t<T(\leq 1)$. Then, we take a supremum over $t\in\left] 0,T \right[$ and we get
  \begin{align*}
  \norm{u_{n+2}-u_{n+1}\mid X_{T}}
  \leq C
  \norm{u_{n+1}-u_{n}\mid X_{T}}
  \qty(
  \norm{u_{n+1}\mid X_{T}}^{\gam -1}
  +\norm{u_{n}\mid X_{T}}^{\gam -1}
  )
\end{align*}
for all $T(\leq 1)$.
  Combining a bound \eqref{bound} with this estimate, we obtain
  \begin{align*}
    \norm{u_{n+2}-u_{n+1}\mid X_{T}}
    \leq C \del^{\gam -1}
    \norm{u_{n+1}-u_{n}\mid X_{T}}.
  \end{align*}
\end{proof}
\begin{pfmain}
  Take $\del$ and $T$ so small that
  \begin{align*}
    \norm{u_{n+2}-u_{n+1}\mid X_{T}}\leq \frac{1}{2}
    \norm{u_{n+1}-u_{n}\mid X_{T}}
  \end{align*}
  for $n=0,1,...$.
  We see that $u_n$ converges in $X_T$. Set $u$ as a limit of $u_n$ in $X_T$. Clearly $u$ is a mild solution of problem \eqref{eq}.
  \qed
\end{pfmain}
\subsection{Convergence to initial data and class of mild solutions}
Let us see that the mild solution $u$ we obtained in Theorem \ref{maintheorem} converges to initial data in the weak-$\ast$ topolpgy of $B^{s-N/p}_{\infty,\infty}$.
\begin{lem}\label{BesovSmoothing}
  For all $s\in\R$ and for all $\psi\in B^{s}_{1,1}$,
  \begin{align*}
    \norm{\cF^{-1}\exp(-t^{\al}\theta\abs{\xi}^2)\cF\psi\mid B^{s}_{1,1}}
    \leq
    \norm{\psi\mid B^{s}_{1,1}}.
  \end{align*}
\end{lem}
\begin{proof}
  From Young inequality, we see that
  \begin{align*}
    &\norm{\cF^{-1}\exp(-t^{\al}\theta\abs{\xi}^2)\cF\psi\mid B^{s}_{1,1}}
    \\
    &=
    \norm{\cF^{-1}\phi_{(0)}(\xi)\exp(-t^{\al}\theta\abs{\xi}^2)\cF\psi\mid L^{1}}
    +
    \sum_{j=1}^{\infty}2^{sj}
    \norm{\cF^{-1}\phi_{j}(\xi)\exp(-t^{\al}\theta\abs{\xi}^2)\cF\psi\mid L^{1}}
    \\
    &=
    \sum_{j=0}^{\infty}2^{sj}
    \norm{\qty(\cF^{-1}\phi_{j}(\xi)\cF\psi)\ast\qty(\cF^{-1}\exp(-t^{\al}\theta\abs{\xi}^2))\mid L^{1}}
    \\
    &\leq
    \norm{\cF^{-1}\exp(-t^{\al}\theta\abs{\xi}^2)\mid{L^{1}}}
    \sum_{j=0}^{\infty}2^{sj}
    \norm{\cF^{-1}\phi_{j}(\xi)\cF\psi\mid L^{1}}
    =
    \norm{\psi\mid B^{s}_{1,1}}.
  \end{align*}
\end{proof}
\begin{thm}\label{thm:convtoini}
  Let $s$ and $p$ satisfy the same conditions as in Theorem \ref{maintheorem}. In addition, we assume that $s$ satisfies $s>-2/\gam$. Then, for all $\psi\in B^{N/p-s}_{1,1}$,
  \begin{align*}
    \left< u(x,t), \psi \right>\to\left< \mu, \psi \right>\quad\text{as }t\to 0.
  \end{align*}
\end{thm}
\begin{proof}
  We take $\psi\in B^{N/p-s}_{1,1}$. From \eqref{eq:int}, we have
  \begin{align*}
    \abs{\left< u(t), \psi \right>-\left< \mu, \psi \right>}
    \leq
    \abs{\left< P_{\al}(t)\mu-\mu, \psi \right>}
    +
    \abs{\left< \int_{0}^{t}(t-\tau)^{\al-1}S_{\al}(t-\tau)|u|^{\gam-1}u(\tau)\,d\tau, \psi \right>}.
  \end{align*}
  For the first term, by using Lemma \ref{wright}-\eqref{wright-a}, we see that
  \begin{align*}
    \abs{\left< P_{\al}(t)\mu-\mu, \psi \right>}
    &=
    \abs{\left< \int_{0}^{\infty}\Phi_{\al}(\theta)e^{t^{\al}\theta\Del}\mu \,d\theta, \psi \right>-\left<\mu,\psi\right>}
    \\&=
    \abs{\int_{0}^{\infty}\Phi_{\al}(\theta)\left<\cF^{-1}\exp(-t^{\al}\theta\abs{\xi}^2)\cF\mu , \psi \right>d\theta-\left<\mu,\psi\right>\int_{0}^{\infty}\Phi_{\al}(\theta)\,d\theta}
    \\
    &=
    \abs{
    \int_{0}^{\infty}\Phi_{\al}(\theta)\left<\mu ,\cF^{-1}\qty(\exp(-t^{\al}\theta\abs{\xi}^2)-1)\cF\psi\right>\,d\theta}
    \\
    &\leq
    \int_{0}^{\infty}\Phi_{\al}(\theta)\abs{\left< \mu, \cF^{-1}\qty(\exp(-t^{\al}\theta\abs{\xi}^2)-1)\cF\psi \right>}\,d\theta
    \\
    &\leq\norm{\mu\mid B^{s-N/p}_{\infty,\infty}}
    \int_{0}^{\infty}\Phi_{\al}(\theta)\norm{\cF^{-1}\qty(\exp(-t^{\al}\theta\abs{\xi}^2)-1)\cF\psi\mid B^{N/p-s}_{1,1}}\,d\theta.
  \end{align*}
  By using Lemma \ref{BesovSmoothing}, the integrand is estimated as follows.
  \begin{align*}
    &\norm{\cF^{-1}\qty(\exp(-t^{\al}\theta\abs{\xi}^2)-1)\cF\psi\mid B^{N/p-s}_{1,1}}
    \\
    \leq\:&
    \norm{\cF^{-1}\exp(-t^{\al}\theta\abs{\xi}^2)\cF\psi\mid B^{N/p-s}_{1,1}}
    +
    \norm{\psi\mid B^{N/p-s}_{1,1}}
    \leq
    2\norm{\psi\mid B^{N/p-s}_{1,1}}
    <\infty.
  \end{align*}
  In the light of Lebesgue convergence theorem, we conclude that
  \begin{align*}
    &\lim_{t\to 0}
    \abs{\left< P_{\al}(t)\mu-\mu, \psi \right>}
    \\
    \leq\:&
    \norm{\mu\mid B^{s-N/p}_{\infty,\infty}}
    \int_{0}^{\infty}
    \Phi_{\al}(\theta)
    \lim_{t\to 0}
    \norm{\cF^{-1}\qty(\exp(-t^{\al}\theta\abs{\xi}^2)-1)\cF\psi\mid B^{N/p-s}_{1,1}}\,d\theta=0.
  \end{align*}
  (For the last limit, see Lemma 3.3 of \cite{KozonoYamazaki}.)
  For the second term, we take arbitrary $\ep>0$ and take $\tilde{\psi}\in\cS(\R^{N})$ such that $\norm{\psi-\tilde{\psi}\mid B^{N/p-s}_{1,1}}<\ep$.
  Then, we have
  \begin{align}\label{eq:psitilde}
    \abs{\left< \int_{0}^{t}(t-\tau)^{\al-1}S_{\al}(t-\tau)|u|^{\gam-1}u(\tau)\,d\tau, \psi-\tilde{\psi} \right>}
    \leq
    &C\norm{\psi-\tilde{\psi}\mid B^{N/p-s}_{1,1}}\notag
    \\
    <&C\ep
  \end{align}
  in the same manner as in the proof of Lemma \ref{zenkashiki}. Next, we have the estimate
  \begin{align*}
    &\abs{\left< \int_{0}^{t}(t-\tau)^{\al-1}S_{\al}(t-\tau)|u|^{\gam-1}u(\tau)\,d\tau, \tilde{\psi} \right>}
    \\
    &\leq
    \int_{0}^{t}(t-\tau)^{\al-1}\abs{\left< S_{\al}(t-\tau)|u|^{\gam-1}u(\tau), \tilde{\psi} \right>}\,d\tau
    \\
    &\leq
    \int_{0}^{t}(t-\tau)^{\al-1}\al \int_{0}^{\infty}
    \theta\Phi_{\al}(\theta)
    \abs{\left<e^{(t-\tau)^{\al}\theta\Del}|u|^{\gam-1}u(\tau), \tilde{\psi} \right>}\,d\theta d\tau
    \\
    &=
    \al \int_{0}^{\infty}
    \theta\Phi_{\al}(\theta)
    \int_{0}^{t}(t-\tau)^{\al-1}
    \abs{\left< |u|^{\gam-1}u(\tau), \cF^{-1}\exp\qty(-(t-\tau)^{\al}\theta\abs{\xi}^2)\cF\tilde{\psi} \right>}
    \,d\tau d\theta.
  \end{align*}
  Note that there is a relation
  \begin{align*}
    M^{p/\gam}_{q/\gam}\subset N^{0}_{p/\gam,q/\gam,\infty}\subset B^{-N\gam/p}_{\infty,\infty}.
  \end{align*}
  Combining this with Lemma \ref{BesovSmoothing}, we obtain
  \begin{align}\label{eq:tzeropsi}
    &\al \int_{0}^{\infty}
    \theta\Phi_{\al}(\theta)
    \int_{0}^{t}(t-\tau)^{\al-1}
    \abs{\left< |u|^{\gam-1}u(\tau), \cF^{-1}\exp\qty(-(t-\tau)^{\al}\theta\abs{\xi}^2)\cF\tilde{\psi} \right>}
    \,d\tau d\theta\notag
    \\
    &\leq
    \al \int_{0}^{\infty}
    \theta\Phi_{\al}(\theta)
    \int_{0}^{t}(t-\tau)^{\al-1}
    \norm{|u|^{\gam-1}u(\cdot,\tau)\mid B^{-N\gam/p}_{\infty,\infty}}
    \notag
    \\
    &\qquad\qquad\qquad\qquad\qquad\times
    \norm{\cF^{-1}\exp\qty(-(t-\tau)^{\al}\theta\abs{\xi}^2)\cF\tilde{\psi}\mid B^{N\gam/p}_{1,1}}
    \,d\tau d\theta\notag
    \\
    &\leq
    \al\norm{\tilde{\psi}\mid B^{N\gam/p}_{1,1}}\int_{0}^{\infty}
    \theta\Phi_{\al}(\theta)
    \int_{0}^{t}(t-\tau)^{\al-1}
    \norm{|u|^{\gam-1}u(\cdot,\tau)\mid M^{p/\gam}_{q/\gam}}
    \,d\tau d\theta\notag
    \\
    &\leq
    \al\norm{\tilde{\psi}\mid B^{N\gam/p}_{1,1}}
    \norm{u\mid X_{T}}^{\gam}
    \int_{0}^{\infty}
    \theta\Phi_{\al}(\theta)
    \int_{0}^{t}(t-\tau)^{\al-1}\tau^{s\al\gam/2}
    \,d\tau d\theta\notag
    \\
    &=
    \al\norm{\tilde{\psi}\mid B^{N\gam/p}_{1,1}}
    \norm{u\mid X_{T}}^{\gam}
    \int_{0}^{t}(t-\tau)^{\al-1}\tau^{s\al\gam/2}
    \,d\tau
    \xrightarrow{t\rightarrow 0} 0.
  \end{align}
  Gathering \eqref{eq:psitilde} and \eqref{eq:tzeropsi}, we get the conclusion.
\end{proof}
\begin{rmk}
  With respect to the estimate of the first term, the additional assumption $s>-2/\gam$ is unnecessary.
\end{rmk}


Next, we verify that the mild soluton $u$ we obtained in Theorem \ref{maintheorem} becomes a strong solution as a Besov--Morrey space valued function.
\begin{thm}\label{strongsolution}
  The solution $u$ to problem \eqref{eq} is continuous as a Besov--Morrey space valued function. That is to say,
  \begin{align*}
    u\in C\qty(\:\left]0,T\right];N^{s}_{p,q,\infty}(\R^N)).
  \end{align*}
  Here, $s,\:p,\:q$ and $\mu$ satisfies the assumptions of Theorem \ref{maintheorem}.
\end{thm}
To prove this, we use the following lemma.
\begin{lem}\label{lemmaforstrongsol}
  The map $t\mapsto P_{\al}(t)\mu$ is continuous in Besov--Morrey norm. Namely,
  \begin{align*}
    \lim_{h\rightarrow 0}\BMnorm{P_{\al}(t+h)\mu-P_{\al}(t)\mu}{s}{p}{q}{\infty}
    =0
  \end{align*}
  holds for all $t>0$. Here, $s,\:p,\:q$ and $\mu$ satisfies the assumptions of Theorem \ref{maintheorem}.
\end{lem}
\begin{proof}
  The Minkowski inequality gives
  \begin{align*}
    \BMnorm{P_{\al}(t+h)\mu-P_{\al}(t)\mu}{s}{p}{q}{\infty}
    \leq
    \int_{0}^{\infty}\Phi_{\al}(\theta)
    \BMnorm{e^{(t+h)^{\al}\theta\Del}\mu-e^{t^{\al}\theta\Del}\mu}
    {s}{p}{q}{\infty}
    \,d\theta.
  \end{align*}
  By using Lemma \ref{smoothingforGauss}, we see
  \begin{align*}
    &\BMnorm{e^{(t+h)^{\al}\theta\Del}\mu-e^{t^{\al}\theta\Del}\mu}
    {s}{p}{q}{\infty}
    \leq
    \BMnorm{e^{(t+h)^{\al}\theta\Del}\mu}
    {s}{p}{q}{\infty}
    \:+\:
    \BMnorm{e^{t^{\al}\theta\Del}\mu}
    {s}{p}{q}{\infty}
    \\
    &\leq
    C\BMnorm{\mu}
    {s}{p}{q}{\infty}
    \:+\:
    C\BMnorm{\mu}{s}{p}{q}{\infty}
    =C\BMnorm{\mu}{s}{p}{q}{\infty}.
  \end{align*}
  By this estimate, Lemma \ref{wright}-\eqref{wright-a} and dominated convergence theorem, we get
  \begin{align*}
    \lim_{h\rightarrow 0}\BMnorm{P_{\al}(t+h)u-P_{\al}(t)u}{s}{p}{q}{\infty}
    &\leq
    \int_{0}^{\infty}\Phi_{\al}(\theta)
    \lim_{h\rightarrow 0}
    \BMnorm{e^{(t+h)^{\al}\theta\Del}\mu-e^{t^{\al}\theta\Del}\mu}
    {s}{p}{q}{\infty}
    \,d\theta.
  \end{align*}
  Here, for $t>0$ and $\theta>0$, we calculate
  \begin{align*}
    \BMnorm{e^{(t+h)^{\al}\theta\Del}\mu-e^{t^{\al}\theta\Del}\mu}
    {s}{p}{q}{\infty}
    &\leq
    \BMnorm{e^{(t+h)^{\al}\theta\Del}\mu-e^{t^{\al}\theta\Del}\mu}
    {0}{p}{q}{\infty}
    \\
    \leq
    \Mnorm{e^{(t+h)^{\al}\theta\Del}\mu-e^{t^{\al}\theta\Del}\mu}
    {p}{q}
    &\leq
    \inftynorm{e^{(t+h)^{\al}\theta\Del}\mu-e^{t^{\al}\theta\Del}\mu}
    \xrightarrow{h\rightarrow 0} 0.
  \end{align*}
  This yields the conclusion of the lemma.
\end{proof}
\begin{rmk}
  The closed operator $-\Del$ is not densely defined on $N^{s}_{p,q,r}$ when $p\neq q$ or $r=\infty$. Thus, we can't expect that $P_{\al}(t)\mu$ is continuous at $t=0$.
\end{rmk}
\begin{rmk}\label{rmkforstrongsol}
  By similar arguments, we see that
  \begin{align*}
    \lim_{h\rightarrow 0}
    \BMnorm{\qty(S_{\al}(t+h-\tau)-S_{\al}(t-\tau))|u|^{\gam-1}u(\cdot,\tau)}{s}{p}{q}{\infty}
    =0
  \end{align*}
  for all $t>0$ and $\tau\in\left]0,t\right[$.
\end{rmk}
\begin{pfstrong}
  {\bf Step 1.}
  First, we show the right continuity of $u(t)$. Take $h>0$ and calculate as follows.
  \begin{align*}
    0&\leq\BMnorm{u(\cdot,t+h)-u(\cdot,t)}{s}{p}{q}{\infty}
    \\
    &\leq\BMnorm{P_{\al}(t+h)\mu-P_{\al}(t)\mu}{s}{p}{q}{\infty}
    \\
    &\quad+\int_{t}^{t+h}
    (t+h-\tau)^{\al-1}
    \BMnorm{S_{\al}(t+h-\tau)|u|^{\gam-1}u(\cdot,\tau)}{s}{p}{q}{\infty}
    \,d\tau
    \\
    &\quad+
    \int_{0}^{t}
    \abs{(t+h-\tau)^{\al-1}-(t-\tau)^{\al-1}}\times
    \BMnorm{S_{\al}(t-\tau)|u|^{\gam-1}u(\cdot,\tau)}{s}{p}{q}{\infty}
    \,d\tau
    \\
    &\quad+
    \int_{0}^{t}
    (t+h-\tau)^{\al-1}
    \BMnorm{
    \left(S_{\al}(t+h-\tau)-S_{\al}(t-\tau)\right)|u|^{\gam-1}u(\cdot,\tau)
    }{s}{p}{q}{\infty}
    \,d\tau
    \\
    &\eqqcolon
    I_{0}+I_{1}+I_{2}+I_{3}.
  \end{align*}
  From Lemma \ref{lemmaforstrongsol}, $I_{0}\xrightarrow{h\rightarrow 0}0$. Next, by using Proposition \ref{inclusion}, Proposition \ref{Sobolevtypeineq},
  Lemma \ref{smoothingforSalpha} and Proposition \ref{rmkformorreynorm}.\ref{rmkformorreynorm-2}, we see that
  \begin{align*}
    I_{1}
    &\leq C
    \int_{t}^{t+h}
    (t+h-\tau)^{\al-1}
    \BMnorm{S_{\al}(t+h-\tau)|u|^{\gam-1}u(\cdot,\tau)}
    {s}{p}{q}{1}
    \,d\tau
    \\
    &\leq C
    \int_{t}^{t+h}
    (t+h-\tau)^{\al-1}
    \BMnorm{S_{\al}(t+h-\tau)|u|^{\gam-1}u(\cdot,\tau)}
    {s+\frac{N(\gam -1)}{p}}{p/\gam}{q/\gam}{1}
    \,d\tau
    \\
    &\leq C
    \int_{t}^{t+h}
    (t+h-\tau)^{\al-1}
    \qty(1+(t+h-\tau))^{-\al s/2 -N(\gam-1)\al/2p}
    \BMnorm{|u|^{\gam-1}u(\cdot,\tau)}
    {0}{p/\gam}{q/\gam}{\infty}
    \,d\tau
    \\
    &\leq C
    \int_{t}^{t+h}
    (t+h-\tau)^{\al-1-\al s/2 -N(\gam-1)\al/2p}\:
    \tau^{s\al\gam/2}\tau^{-s\al\gam/2}
    \Mnorm{u(\cdot,\tau)}{p}{q}^{\gam}
    \,d\tau
    \\
    &\leq C
    t^{s\al\gam/2}\xtnorm{u}^{\gam}
    \int_{t}^{t+h}
    (t+h-\tau)^{\al-1-\al s/2 -N(\gam-1)\al/2p}
    \,d\tau
    \\
    &= C
    t^{s\al\gam/2}\xtnorm{u}^{\gam}\cdot
    h^{\al-\al s/2 -N(\gam-1)\al/2p}
    \xrightarrow{h\rightarrow 0}0.
  \end{align*}
  We note that the relation $\al-\al s/2-N(\gam-1)\al/2p>0$ holds from assumptions $s<0$ and
   $s\geq N/p-2/(\gam -1)$.
  For $I_{2}$, we see that the integrand is estimated as follows.
  \begin{align*}
    &|(t+h-\tau)^{\al-1}-(t-\tau)^{\al-1}|
    \BMnorm{S_{\al}(t-\tau)|u|^{\gam-1}u(\cdot,\tau)}{s}{p}{q}{\infty}
    \\
    &\leq C
    \left((t+h-\tau)^{\al-1}+(t-\tau)^{\al-1}\right)
    (t-\tau)^{-\al s/2 -N(\gam-1)\al/2p}\tau^{\frac{s\al\gam}{2}}\norm{u\mid X_{T}}^{\gam}
    \\
    &\leq C
    (t-\tau)^{\al-1-\al s/2 -N(\gam-1)\al/2p}\tau^{\frac{s\al\gam}{2}}.
  \end{align*}
  Thus, the Lebesgue convergence theorem yields
  \begin{align*}
    \lim_{h\to 0}I_{2}
    =\int_{0}^{t}\lim_{h\to 0}
    |(t+h-\tau)^{\al-1}-(t-\tau)^{\al-1}|
    \BMnorm{S_{\al}(t-\tau)|u|^{\gam-1}u(\cdot,\tau)}{s}{p}{q}{\infty}
    \,d\tau
    =0.
  \end{align*}
To estimate $I_{3}$, we see that
\begin{align*}
  &(t+h-\tau)^{\al-1}
  \BMnorm{\left(
  S_{\al}(t+h-\tau)-S_{\al}(t-\tau)\right)|u|^{\gam-1}u(\cdot,\tau)
  }{s}{p}{q}{\infty}
  \\
  &\leq C
  (t-\tau)^{\al-1}
  \Bigl(
  \BMnorm{
  S_{\al}(t+h-\tau)|u|^{\gam-1}u(\cdot,\tau)
  }{s}{p}{q}{\infty}
  \\
  &\qquad\qquad\qquad\qquad\qquad
  +
  \BMnorm{
  S_{\al}(t-\tau)|u|^{\gam-1}u(\cdot,\tau)
  }{s}{p}{q}{\infty}
  \Bigr)
  \\
  &\leq C
  (t-\tau)^{\al-1}
  \left(
  (t+h-\tau)^{-\al s/2 -N(\gam-1)\al/2p}+(t-\tau)^{-\al s/2 -N(\gam-1)\al/2p}
  \right)
  \tau^{\frac{s\al\gam}{2}}\norm{u\mid X_{T}}^{\gam}
  \\
  &\leq C
  (t-\tau)^{\al-1-\al s/2 -N(\gam-1)\al/2p}
  \tau^{\frac{s\al\gam}{2}}.
\end{align*}
Finally, the Lebesgue convergence theorem and Remark \ref{rmkforstrongsol} give
\begin{align*}
  \lim_{h\to 0}I_{3}
  &=
  \int_{0}^{t}\lim_{h\to 0}
  (t+h-\tau)^{\al-1}
  \BMnorm{\left(
  S_{\al}(t+h-\tau)-S_{\al}(t-\tau)\right)|u|^{\gam-1}u(\cdot,\tau)
  }{s}{p}{q}{\infty}
  \,d\tau
  \\
  &=0.
\end{align*}
{\bf Step 2.} Next, we check the left continuity of $u(t)$. Take $h>0$ and calculate as follows.
\begin{align*}
  0&\leq\BMnorm{u(\cdot,t)-u(\cdot,t-h)}{s}{p}{q}{\infty}
  \\
  &\leq\BMnorm{P_{\al}(t)\mu-P_{\al}(t-h)\mu}{s}{p}{q}{\infty}
  \\
  &\quad+\int_{t-h}^{t}
  (t-\tau)^{\al-1}
  \BMnorm{S_{\al}(t-\tau)|u|^{\gam-1}u(\cdot,\tau)}{s}{p}{q}{\infty}
  \,d\tau
  \\
  &\quad+
  \int_{0}^{t-h}
  \abs{(t-h-\tau)^{\al-1}-(t-\tau)^{\al-1}}\times
  \BMnorm{S_{\al}(t-\tau)|u|^{\gam-1}u(\cdot,\tau)}{s}{p}{q}{\infty}
  \,d\tau
  \\
  &\quad+
  \int_{0}^{t-h}
  (t-h-\tau)^{\al-1}
  \BMnorm{
  \qty(S_{\al}(t-h-\tau)-S_{\al}(t-\tau))|u|^{\gam-1}u(\cdot,\tau)
  }{s}{p}{q}{\infty}
  \,d\tau
  \\
  &\eqqcolon
  J_{0}+J_{1}+J_{2}+J_{3}.
\end{align*}
From Lemma \ref{lemmaforstrongsol}, $J_{0}\xrightarrow{h\rightarrow 0}0$ for $t>0$.
Mimicking the estimate of $I_{1}$, we get
\begin{align*}
  J_{1}
  &\leq
  C(t-h)^{s\al\gam/2}\norm{u\mid X_{T}}^{\gam}\int_{t-h}^{t}
  (t-\tau)^{\al-1-\al s/2-N(\gam-1)\al/2p}
  \,d\tau
  \\
  &\leq C
  (t-h)^{s\al\gam/2}
  h^{\al-\al s/2 -N(\gam-1)\al/2p}
  \xrightarrow{h\rightarrow 0}0.
\end{align*}
In order to estimate $J_{2}$, we take arbitrary $\ep\in\left]0,t-h\right[$ and we calculate as follows.
\begin{align*}
  J_{2}
  &\leq
  \int_{0}^{t-h}
  \qty((t-h-\tau)^{\al-1}-(t-\tau)^{\al-1})
  \BMnorm{S_{\al}(t-\tau)|u|^{\gam-1}u(\cdot,\tau)}{s}{p}{q}{\infty}
  \,d\tau
  \\
  &\leq
  C\int_{\ep}^{t-h}
  \qty((t-h-\tau)^{\al-1}-(t-\tau)^{\al-1})
  (t-\tau)^{-\al s/2-N(\gam-1)\al/2p}\tau^{s\al\gam/2}
  \,d\tau
  \\
  &\quad+C\int_{0}^{\ep}
  \qty((t-h-\tau)^{\al-1}-(t-\tau)^{\al-1})
  (t-\tau)^{-\al s/2-N(\gam-1)\al/2p}\tau^{s\al\gam/2}
  \,d\tau
  \\
  &\leq C\ep^{s\al\gam/2}\int_{0}^{t-h}
  \qty((t-h-\tau)^{\al-\al s/2-N(\gam-1)\al/2p-1}-(t-\tau)^{\al-\al s/2-N(\gam-1)\al/2p-1})
  \,d\tau
  \\
  &\quad+
  C\int_{0}^{\ep}\tau^{s\al\gam/2}
  \,d\tau
  \\
  &=
  C\ep^{s\al\gam/2}\qty{
  h^{\al-\al s/2-N(\gam-1)\al/2p}-t^{\al-\al s/2-N(\gam-1)\al/2p}+(t-h)^{\al-\al s/2-N(\gam-1)\al/2p}
  }
  \\
  &\quad+C\ep^{s\al\gam/2 +1}.
\end{align*}
This calculation gives
\begin{align*}
  0\leq
  \limsup_{h\rightarrow 0} J_{2}
  \leq
  C\ep^{s\al\gam/2 +1}.
\end{align*}
By arbitrariness of $\ep>0$, we get $J_{2}\xrightarrow{h\rightarrow 0}0$.
To see the convergence of $J_{3}$, we take arbitrary $\ep\in\left]0,t\right[$. Then, we take arbitrary $h\in\left]0,\ep\right[$ and decompose $J_{3}$ as
\begin{align*}
  J_{3}
  &=\int_{0}^{t-\ep}
  (t-h-\tau)^{\al-1}
  \BMnorm{
  \qty(S_{\al}(t-h-\tau)-S_{\al}(t-\tau))|u|^{\gam-1}u(\cdot,\tau)
  }{s}{p}{q}{\infty}
  \,d\tau
  \\
  &\quad+\int_{t-\ep}^{t-h}
  (t-h-\tau)^{\al-1}
  \BMnorm{
  \qty(S_{\al}(t-h-\tau)-S_{\al}(t-\tau))|u|^{\gam-1}u(\cdot,\tau)
  }{s}{p}{q}{\infty}
  \,d\tau
  \\
  &\eqqcolon J_{31}+J_{32}.
\end{align*}
By the same reasoning as in the convergence of $I_{3}$, we get $J_{31}\xrightarrow{h\rightarrow 0}0$.
For $J_{32}$,
\begin{align*}
  J_{32}&\leq
  \int_{t-\ep}^{t-h}
  (t-h-\tau)^{\al-1}
  \biggl(
  \norm{S_{\al}(t-h-\tau)|u|^{\gam-1}u(\cdot,\tau)\mid N^{s}_{p,q,\infty}}
  \\
  &\qquad\qquad\qquad\qquad\qquad\qquad+
  \norm{S_{\al}(t-\tau)|u|^{\gam-1}u(\cdot,\tau)\mid N^{s}_{p,q,\infty}}
  \biggr)\,d\tau
  \\
  &\leq C
  \int_{t-\ep}^{t-h}
  (t-h-\tau)^{\al-1}
  \biggl(
  (t-h-\tau)^{-\al s/2-N(\gam-1)\al/2p}\tau^{s\al\gam/2}
  \\
  &\qquad\qquad\qquad\qquad\qquad\qquad+
  (t-\tau)^{-\al s/2-N(\gam-1)\al/2p}\tau^{s\al\gam/2}
  \biggr)\,d\tau
  \\
  &\leq 2C(t-\ep)^{s\al\gam/2}
  \int_{t-\ep}^{t-h}
  (t-h-\tau)^{\al-1-\al s/2-N(\gam-1)\al/2p}
  \,d\tau
  \\
  &=
  C(t-\ep)^{s\al\gam/2}
  (\ep-h)^{\al-\al s/2-N(\gam-1)\al/2p}.
\end{align*}
Therefore,
\begin{align*}
  0\leq
  \limsup_{h\rightarrow 0} J_{32}
  \leq
  C(t-\ep)^{s\al\gam/2}
  \ep^{\al-\al s/2-N(\gam-1)\al/2p}.
\end{align*}
By arbitrariness of $\ep>0$, we get $J_{32}\xrightarrow{h\rightarrow 0}0$
and this completes the proof.
  \qed
\end{pfstrong}
\section{Global existence results for small data}\label{sec:globalsol}
In this section, we prove the existence of a global-in-time mild solution to problem \eqref{eq} for small initial data. 
We begin this section by gathering some properties of homogeneous Besov--Morrey spaces \cite{KozonoYamazaki}.
\begin{prop}\label{HSobolev}
  {\bf (Sobolev embedding)}
  Let $1\leq q\leq p<\infty$, $r\in[1,\infty]$ and $s\in\R$.
  Then, for all $l\in\left]0,1\right[$,
  \begin{align*}
    \cN^{s}_{p,q,r}\subset \cN^{s-N(1-l)/p}_{p/l,q/l,r}.
  \end{align*}
\end{prop}
\begin{prop}\label{Hinclusion}
  For $1\leq q\leq p<\infty$, we have the inclusion relations
  \begin{align*}
    \cN_{p,q,1}^{0}\subset N_{p,q,1}^{0},
    \quad
    \cN_{p,q,1}^{0}\subset\cM^{p}_{q}\subset\cN^{0}_{p,q,\infty}.
  \end{align*}
\end{prop}
\begin{lem}\label{HGauss}
  Let $s\leq\sigma,\: 1\leq q\leq p<\infty$ and $r\in[1,\infty]$.
  Then, there exists $C>0$ such that for all $u\in N^{s}_{p,q,r}$, the estimate
  \begin{align*}
    \norm{e^{t\Del}u\mid\cN^{\sig}_{p,q,r}}
    \leq C
    t^{\frac{s-\sigma}{2}}
    \norm{u\mid\cN^{s}_{p,q,r}}
    \quad\text{for}\quad t>0
  \end{align*}
  holds. Furthermore, if $s<\sigma$, the estimate
  \begin{align*}
    \norm{e^{t\Del}u\mid\cN^{\sig}_{p,q,1}}
    \leq C
    t^{\frac{s-\sigma}{2}}
    \norm{u\mid\cN^{s}_{p,q,\infty}}
    \quad\text{for}\quad t>0
  \end{align*}
  holds.
\end{lem}
As before, we get estimates for the operators $S_{\al}(t)$ and $P_{\al}(t)$.
\begin{lem}\label{HSalpha}
  Let $s\leq\sigma,\: 1\leq q\leq p<\infty,\: r\in[1,\infty]$ and $4>\sigma -s$.
  Then, there exists $C>0$ such that for all $u\in \cN^{s}_{p,q,r}$, the estimate
  \begin{align*}
    \norm{S_{\al}(t)u\mid\cN^{\sig}_{p,q,r}}
    \leq C
    t^{\frac{(s-\sigma)\al}{2}}
    \norm{u\mid\cN^{s}_{p,q,r}}
    \quad\text{for}\quad t>0
  \end{align*}
  holds. Furthermore, if $s<\sigma$, the estimate
  \begin{align*}
    \norm{S_{\al}(t)u\mid\cN^{\sig}_{p,q,1}}
    \leq C
    t^{\frac{(s-\sigma)\al}{2}}
    \norm{u\mid\cN^{s}_{p,q,\infty}}
    \quad\text{for}\quad t>0
  \end{align*}
  holds.
\end{lem}
\begin{proof}
  Lemma \ref{HGauss}, Lemma \ref{wright}-\eqref{wright-b} and Minkowski inequality give
  \begin{align*}
    \norm{S_{\al}(t)u\mid\cN^{\sig}_{p,q,r}}
    &=\norm{\al\int_{0}^{\infty}\theta\Phi_{\al}(\theta)e^{t^{\al}\theta\Del}u(x)d\theta\mid\cN^{\sig}_{p,q,r}}
    \\
    &\leq
    \al\int_{0}^{\infty}\theta\Phi_{\al}(\theta)
    \norm{e^{t^{\al}\theta\Del}u\mid\cN^{\sig}_{p,q,r}}
    d\theta
    \\
    &\leq
    C \al\int_{0}^{\infty}
    \theta\Phi_{\al}(\theta)
    (t^{\al}\theta)^{\frac{s-\sigma}{2}}
    \norm{u\mid\cN^{s}_{p,q,r}}
    d\theta
    \\
    &= C\norm{u\mid\cN^{s}_{p,q,r}}
    t^{\frac{(s-\sigma)\al}{2}}
    \int_{0}^{\infty}
    \theta^{1+\frac{s-\sigma}{2}}
    \Phi_{\al}(\theta)d\theta
    \\
    &\leq C\norm{u\mid\cN^{s}_{p,q,r}}
    t^{\frac{(s-\sigma)\al}{2}}.
  \end{align*}
  Note that $\:\:1+\frac{s-\sigma}{2}>-1\:\:$ if $\:\:4>\sigma-s$.
\end{proof}
\begin{lem}\label{HPalpha}
  Let $s\leq\sigma,\: 1\leq q\leq p<\infty,\: r\in[1,\infty]$ and $2>\sigma -s$.
  Then, there exists $C>0$ such that for all $u\in \cN^{s}_{p,q,r}$, the estimate
  \begin{align*}
    \norm{P_{\al}(t)u\mid\cN^{\sig}_{p,q,r}}
    \leq C
    t^{\frac{(s-\sigma)\al}{2}}
    \norm{u\mid\cN^{s}_{p,q,r}}
    \quad\text{for}\quad t>0
  \end{align*}
  holds. Furthermore, if $s<\sigma$, the estimate
  \begin{align*}
    \norm{P_{\al}(t)u\mid\cN^{\sig}_{p,q,1}}
    \leq C
    t^{\frac{(s-\sigma)\al}{2}}
    \norm{u\mid\cN^{s}_{p,q,\infty}}
    \quad\text{for}\quad t>0
  \end{align*}
  holds.
\end{lem}
\begin{proof}
  The proof is similar to that of Lemma \ref{HSalpha}.
\end{proof}
We begin the proof of Theorem \ref{globalsol}.
We set $\beta\coloneqq\frac{\al}{\gam-1}-\frac{\al N}{2p}$ and we write
\begin{align*}
  &X\coloneqq
  \qty{
  u(x,t) : \text{Lebesgue measurable in } \R^{N}\times\R_{>0}
  \mid
  \norm{u\mid X}<\infty
  },
\end{align*}
where
\begin{align*}
  \norm{u\mid X}
  \coloneqq\sup_{0<t}t^{\beta}\norm{u(\cdot,t)\mid\cM^{p}_{q}}.
\end{align*}
We define $u_{0}\coloneqq P_{\al}(t)\mu$ and $u_{n}\:(n\in\Z_{\geq 1})$ inductively by
\begin{align*}
  u_{n}(t)\coloneqq u_{0}(t)+\int_{0}^{t}(t-\tau)^{\al -1}S_{\al}(t-\tau)|u_{n-1}|^{\gam -1}u_{n-1}(\tau)
  \,d\tau.
\end{align*}
\begin{rmk}\label{aboutbeta}
  The constant $\beta$ satisfies $-\beta\gam>-1$ and $\beta<\al$ if we take $p$ and $\gam$ as in Theorem $\ref{globalsol}$.
\end{rmk}
As before, we prepare the following three lemmata to prove Theorem \ref{globalsol}.
\begin{lem}\label{Hstep1}
  Let $1\leq q\leq p<\infty$ 
  and $\del>0$.
  Then, there exists $C_{4}>0$ such that for all $\mu(x)\in \cN^{N/p-2/(\gam -1)}_{p,q,\infty}$ satisfying  $\onorm{\mu}{\cN^{N/p-2/(\gam -1)}_{p,q,\infty}}\leq\del$, the following inequality
  \[
  \norm{u_{0}\mid X}\leq C_{4}\del
  \]
  holds.
\end{lem}
\begin{proof}
  By using Lemma \ref{Hinclusion} and Lemma \ref{HPalpha}, there exists $C>0$ such that
  \begin{align*}
    \norm{u_{0}\mid\cM^{p}_{q}}
    &=\norm{P_{\al}(t)\mu\mid\cM^{p}_{q}}
    \leq\norm{P_{\al}(t)\mu\mid\cN^{0}_{p,q,1}}
    \leq C t^{-\beta}\norm{\mu\mid\cN^{N/p-2/(\gam -1)}_{p,q,\infty}}.
  \end{align*}
  Seeing Remark \ref{aboutbeta}, we have
  \begin{align*}
    t^{\beta}\norm{u_{0}\mid\cM^{p}_{q}}
    &\leq C t^{\beta-\beta}\norm{\mu\mid\cN^{s}_{p,q,\infty}}
    \leq C \del\quad \text{for }t>0.
  \end{align*}
  Taking sup in $t>0$, we get the claim.
\end{proof}
\begin{lem}\label{Hstep2}
  Let $\gam>1$ and $\gam\leq q\leq p<\infty$. 
  Then, there exists $C_{5}>0$ such that
  \begin{align*}
    \norm{u_{n+1}\mid X}
    &\leq
    \norm{u_{0}\mid X}+
    C_{5}\norm{u_{n}\mid X}^{\gam}
  \end{align*}
  for $n=0,1,...$.
\end{lem}
\begin{proof}
  By using Proposition \ref{Hinclusion}, Proposition \ref{HSobolev},
  Lemma \ref{HSalpha}, Remark \ref{aboutbeta} and Proposition \ref{rmkformorreynorm}.\ref{rmkformorreynorm-2}, we see that
  \begin{align*}
    \norm{u_{n+1}(\cdot,t)-u_{0}(\cdot,t)\mid\cM^{p}_{q}}
    \leq C&
    \norm{u_{n+1}(\cdot,t)-u_{0}(\cdot,t)\mid\cN^{0}_{p,q,1}}
    \\
    \leq C&
    \int_{0}^{t}(t-\tau)^{\al -1}
    \norm{S_{\al}(t-\tau)u_{n}^{\gam}(\cdot,\tau)\mid\cN^{0}_{p,q,1}}
    \,d\tau
    \\
    \leq C&
    \int_{0}^{t}(t-\tau)^{\al -1}
    \norm{S_{\al}(t-\tau)u_{n}^{\gam}(\cdot,\tau)
    \mid\cN^{\frac{N(\gam -1)}{p}}_{p/\gam,q/\gam,1}}
    \,d\tau
    \\
    \leq C&
    \int_{0}^{t}
    (t-\tau)^{\al -1-\frac{N(\gam -1)\al}{2p}}
    \norm{u_{n}^{\gam}(\cdot,\tau)\mid\cN^{0}_{p/\gam,q/\gam,\infty}}
    \,d\tau
    \\
    \leq C&
    \int_{0}^{t}(t-\tau)^{\al -1-\frac{N(\gam -1)\al}{2p}}
    \norm{u_{n}^{\gam}(\cdot,\tau)\mid\cM^{p/\gam}_{q/\gam}}
    \,d\tau
    \\
    \leq C&
    \norm{u_{n}\mid X}^{\gam}
    \int_{0}^{t}(t-\tau)^{\al -1-\frac{N(\gam -1)\al}{2p}}
    \tau^{-\beta\gam}
    \,d\tau
    \\
    =C&
    \norm{u_{n}\mid X}^{\gam}
    t^{\al -\frac{N(\gam -1)\al}{2p}-\beta\gam}
    \int_{0}^{1}(1-\tau)^{\al -1-\frac{N(\gam -1)\al}{2p}}
    \tau^{-\beta\gam}
    \,d\tau.
  \end{align*}
  Therefore, we have
  \begin{align*}
    \norm{u_{n+1}(\cdot,t)-u_{0}(\cdot,t)\mid\cM^{p}_{q}}t^{\beta}
    &\leq C
    \norm{u_{n}\mid X}^{\gam}
    t^{\al -\frac{N(\gam -1)\al}{2p}+\beta(1-\gam)}
    =C
    \norm{u_{n}\mid X}^{\gam}
  \end{align*}
  for all $t>0$. Then, we take a supremum over $t>0$ and we get
  \begin{align*}
    \norm{u_{n+1}-u_{0}\mid X}
    &\leq C
    \norm{u_{n}\mid X}^{\gam}.
  \end{align*}
\end{proof}
By using Lemma \ref{Hstep1} and Lemma \ref{Hstep2} repeatedly, we can check that functions $u_{n}$ have a bound
\begin{align}\label{Hbound}
\sup_{n\in\N}\norm{u_{n}\mid X}\leq M:=2C_{4}\del
\end{align}
if $\del>0$ satisfies $2^{\gam}C_{4}^{\gam-1}C_{5}\del^{\gam-1}\leq 1$.
\begin{lem}\label{Hstep3}
  Let $\gam>1$ and $\gam\leq q\leq p<\infty$. 
  Then, there exists $C>0$ such that
  \begin{align*}
    \norm{u_{n+2}-u_{n+1}\mid X}\leq C\del^{\gam -1}
    \norm{u_{n+1}-u_{n}\mid X}
  \end{align*}
  for $n=0,1,...$.
\end{lem}
\begin{proof}
  By using Proposition \ref{Hinclusion}, Proposition \ref{HSobolev},
  Lemma \ref{HSalpha}, Remark \ref{aboutbeta} and Proposition \ref{rmkformorreynorm}.\ref{rmkformorreynorm-2}, we see that
  \begin{align*}
    &\norm{u_{n+2}(\cdot,t)-u_{n+1}(\cdot,t)\mid\cM^{p}_{q}}
    \\
    \leq C&
    \int_{0}^{t}(t-\tau)^{\al -1}\norm{S_{\al}(t-\tau)\qty(
    |u_{n+1}|^{\gam-1}u_{n+1}(\cdot,\tau)-|u_{n}|^{\gam-1}u_{n}(\cdot,\tau)
    )\mid\cN^{0}_{p,q,1}}
    \,d\tau
    \\
    \leq C&
    \int_{0}^{t}
    (t-\tau)^{\al -1-\frac{N(\gam -1)\al}{2p}}
    \norm{|u_{n+1}|^{\gam-1}u_{n+1}(\cdot,\tau)-|u_{n}|^{\gam-1}u_{n}(\cdot,\tau)\mid\cN^{0}_{p/\gam,q/\gam,\infty}}
    \,d\tau
    \\
    \leq C& \int_{0}^{t}(t-\tau)^{\al -1-\frac{N(\gam -1)\al}{2p}}
    \norm{|u_{n+1}|^{\gam-1}u_{n+1}(\cdot,\tau)-|u_{n}|^{\gam-1}u_{n}(\cdot,\tau)\mid\cM^{p/\gam}_{q/\gam}}
    \,d\tau
    \\
    \leq C& \int_{0}^{t}(t-\tau)^{\al -1-\frac{N(\gam -1)\al}{2p}}
    \norm{u_{n+1}(\cdot,\tau)-u_{n}(\cdot,\tau)\mid\cM^{p}_{q}}
    \\
    &\quad\quad\quad\quad\times
    \qty(
    \norm{u_{n+1}(\cdot,\tau)\mid\cM^{p}_{q}}^{\gam -1}
    +\norm{u_{n}(\cdot,\tau)\mid\cM^{p}_{q}}^{\gam -1}
    )
    \,d\tau
    \\
    \leq C&
    \int_{0}^{t}(t-\tau)^{\al -1-\frac{N(\gam -1)\al}{2p}}\tau^{-\beta\gam}
    \,d\tau
    \norm{u_{n+1}-u_{n}\mid X}\\
    &\quad\quad\quad\quad\times
    \qty(
    \norm{u_{n+1}\mid X}^{\gam -1}
    +\norm{u_{n}\mid X}^{\gam -1}
    ).
  \end{align*}
  Therefore, we have
  \begin{align*}
    \norm{u_{n+2}(\cdot,t)-u_{n+1}(\cdot,t)\mid\cM^{p}_{q}}t^{\beta}
    \leq C&
    \norm{u_{n+1}-u_{n}\mid X}
    t^{\al -\frac{N(\gam -1)\al}{2p}+\beta(1-\gam)}\\
    &\quad\quad\quad\quad\times
    \qty(
    \norm{u_{n+1}\mid X}^{\gam -1}
    +\norm{u_{n}\mid X}^{\gam -1}
    )\\
    =C&
    \norm{u_{n+1}-u_{n}\mid X}
    \qty(
    \norm{u_{n+1}\mid X}^{\gam -1}
    +\norm{u_{n}\mid X}^{\gam -1}
    )
  \end{align*}
  for all $t>0$. Then, we take a supremum over $t>0$ and we get
  \begin{align*}
    \norm{u_{n+2}-u_{n+1}\mid X}
    \leq C&
    \norm{u_{n+1}-u_{n}\mid X}
    \qty(
    \norm{u_{n+1}\mid X}^{\gam -1}
    +\norm{u_{n}\mid X}^{\gam -1}
    ).
  \end{align*}
  Combining a bound \eqref{Hbound} with this estimate, we obtain
  \begin{align*}
    \norm{u_{n+2}-u_{n+1}\mid X}
    \leq C \del^{\gam -1}
    \norm{u_{n+1}-u_{n}\mid X}.
  \end{align*}
\end{proof}
\begin{pfglobal}
  Take $\del$ so small that
  \begin{align*}
    \norm{u_{n+2}-u_{n+1}\mid X}\leq \frac{1}{2}
    \norm{u_{n+1}-u_{n}\mid X}
  \end{align*}
  for $n=0,1,...$.
  We see that $u_n$ converges in $X$. Set $u$ as a limit of $u_n$ in $X$. Clearly $u$ is a mild solution of problem \eqref{eq}.
  \qed
\end{pfglobal}
\begin{ack}
  We would like to thank Professor Tatsuki Kawakami for bringing time fractional heat equations to our attention.
  The first author was supported in part by JST SPRING, Grant Number JPMJSP2108. The second author was supported in part by JSPS KAKENHI Grant-in-Aid for Early-Career Scientists 23K13005.
\end{ack}

\end{document}